\documentclass[11pt]{article}
\usepackage[utf8]{inputenc}
\usepackage[a4paper]{geometry}
\usepackage{mathrsfs}
\usepackage{amsmath}
\usepackage{amsthm}
\usepackage{amssymb}
\usepackage{comment}

\makeatletter
\theoremstyle{definition}
\newtheorem*{problem*}{\protect\problemname}
\theoremstyle{plain}
\newtheorem{theorem}{\protect\theoremname}[section]
\theoremstyle{definition}
\newtheorem{definition}[theorem]{\protect\definitionname}
\theoremstyle{plain}
\newtheorem{lemma}[theorem]{\protect\lemmaname}
\theoremstyle{plain}
\newtheorem{corollary}[theorem]{\protect\corollaryname}
\theoremstyle{plain}
\newtheorem{proposition}[theorem]{\protect\propositionname}
\theoremstyle{plain}
\newtheorem{assumption}[theorem]{\protect\assumptionname}
\theoremstyle{definition}
\newtheorem{example}[theorem]{\protect\examplename}
\theoremstyle{remark}

\theoremstyle{plain}
\newtheorem*{prop*}{\protect\propositionname}

\usepackage{color}

\definecolor{darkmagenta}{RGB}{139,0,139}

\definecolor{refkey}{rgb}{0.9451,0.2706,0.4941}
\definecolor{labelkey}{rgb}{0.9451,0.2706,0.4941}

\numberwithin{equation}{section}

\makeatletter
\def\renewtheorem#1{%
  \expandafter\let\csname#1\endcsname\relax
  \expandafter\let\csname c@#1\endcsname\relax
  \gdef\renewtheorem@envname{#1}
  \renewtheorem@secpar
}
\def\renewtheorem@secpar{\@ifnextchar[{\renewtheorem@numberedlike}{\renewtheorem@nonumberedlike}}
\def\renewtheorem@numberedlike[#1]#2{\newtheorem{\renewtheorem@envname}[#1]{#2}}
\def\renewtheorem@nonumberedlike#1{  
\def\renewtheorem@caption{#1}
\edef\renewtheorem@nowithin{\noexpand\newtheorem{\renewtheorem@envname}{\renewtheorem@caption}}
\renewtheorem@thirdpar
}
\def\renewtheorem@thirdpar{\@ifnextchar[{\renewtheorem@within}{\renewtheorem@nowithin}}
\def\renewtheorem@within[#1]{\renewtheorem@nowithin[#1]}
\makeatother

\theoremstyle{definition}
\renewtheorem{defn}{Definition}[section]
\renewtheorem{assumption}{Assumption}

\usepackage[hang,flushmargin]{footmisc} 

\makeatother

\providecommand{\assumptionname}{Assumption}
\providecommand{\corollaryname}{Corollary}
\providecommand{\definitionname}{Definition}
\providecommand{\examplename}{Example}
\providecommand{\lemmaname}{Lemma}
\providecommand{\problemname}{Problem}
\providecommand{\propositionname}{Proposition}
\providecommand{\remarkname}{Remark}
\providecommand{\theoremname}{Theorem}

\title{Existence of dynamical low rank approximations for random semi-linear evolutionary equations on the maximal interval}
\author{Yoshihito Kazashi\footnotemark[1]\and Fabio Nobile\footnote{Mathematics Institute, CSQI, \'Ecole Polytechnique F\'ed\'erale de Lausanne, Switzerland}}

\begin{document}
\maketitle
\begin{abstract}
An existence result is presented for the dynamical low rank (DLR) approximation
for random semi-linear evolutionary equations. 
The DLR solution approximates the true solution at each time instant by a linear combination of products of deterministic and stochastic basis functions, both of which evolve over time. 
A key to our proof
is to find a suitable equivalent formulation of the original problem.
The so-called Dual Dynamically Orthogonal formulation turns out to
be convenient. 
Based on this formulation, the DLR approximation is recast to an abstract Cauchy problem in a suitable linear space, for which existence and uniqueness of the solution in the maximal interval are established.
\end{abstract}

\section{Introduction}
\begin{sloppypar}
	This paper is concerned with the existence of solutions of the
	so called Dynamical Low Rank Method \cite{Sapsis.T_Lermusiaux_2009_DO,Musharbash.E_etal_2015_SISC,Musharbash.E_Nobile_2018_Dual,Feppon.F_Lermusiaux_2018_geometric,Feppon.F_Lermusiaux_2018_SIREV}
	to a semi-linear random parabolic evolutionary equation. For a separable
	$\mathbb{R}$-Hilbert space $(\mathcal{H},\langle\cdot,\cdot\rangle)$
	and a probability space $(\Omega,\mathscr{F},\mathbb{P})$, let $L^{2}(\Omega;\mathcal{H}):=L_{\mathbb{P}}^{2}(\Omega;\mathcal{H})$
	be the Bochner space of equivalence classes of $\mathcal{H}$-valued
	measurable functions on $\Omega$, with finite second moments. We
	consider the following equation in $L^{2}(\Omega;\mathcal{H})$:
	\begin{equation}
	\frac{\partial u}{\partial t}(t)=\Lambda u(t)+F(u(t)),\quad t>0,\quad\text{with}\ u(0)=u_{0},\label{eq:exact-eq}
	\end{equation}
	with a closed linear operator $\Lambda:D_{\mathcal{H}}(\Lambda)\subset\mathcal{H}\to\mathcal{H}$,
	and a mapping $F\colon L^{2}(\Omega;\mathcal{H})\to L^{2}(\Omega;\mathcal{H})$,
	where the domain $D_{\mathcal{H}}(\Lambda)$ is dense in $\mathcal{H}$.
\end{sloppypar}
Our interest in this paper is a reduced basis method for this problem
called the Dynamically Low Rank (DLR) approximation \cite{Sapsis.T_Lermusiaux_2009_DO,Musharbash.E_etal_2015_SISC,Musharbash.E_Nobile_2018_Dual,Feppon.F_Lermusiaux_2018_geometric,Feppon.F_Lermusiaux_2018_SIREV}.
The idea is to approximate the solution of \eqref{eq:exact-eq} at each time $t>0$ as
a linear combination of products of deterministic and
stochastic basis functions, both of which evolve over time: the
approximate solution is of the form $u_{S}(t)=\boldsymbol{U}^{\top}(t)\boldsymbol{Y}(t)$,
for some positive integer $S\in\mathbb{N}$ called the rank of
the solution, where $\boldsymbol{U}(t)=(U_{1}(t),\dots,U_{S}(t))^{\top}$
are linearly independent in $\mathcal{H}$, and $\boldsymbol{Y}(t)=(Y_{1}(t),\dots,Y_{S}(t))^{\top}$
are linearly independent in the space $L^{2}(\Omega)$ of square-integrable
random variables. We note that both bases depend on the temporal variable
$t$. This dependence is intended to approximate well, with a fixed
(possibly small) rank, the solution of stochastic dynamical
systems such as \eqref{eq:exact-eq}, whose stochastic and spatial
dependence may change significantly in time. Numerical examples and
error analysis suggests the method does indeed work well in
a certain number of practical applications \cite{Sapsis.T_Lermusiaux_2009_DO,Musharbash.E_Nobile_2018_Dual}.

A fundamental open question regarding this approach is the
unique existence of DLR solutions. The DLR approximation is
given as a solution of a system of differential equations, and available
approximation results are built upon the assumption that this
solution exists, e.g.\ \cite{Musharbash.E_etal_2015_SISC,Feppon.F_Lermusiaux_2018_geometric}.
Nonetheless, to the best of our knowledge, the existence---let alone
the uniqueness---of DLR solutions {for an equation of the type \eqref{eq:exact-eq}} is not known. In this paper,
we will establish a unique existence result.

A difficulty in proving the existence is the fact that the solution
propagates in an infinite-dimensional manifold,
and that we have an unbounded operator in the equation. Indeed, the DLR
equations are derived so that the aforementioned approximation $u_{S}$
keeps the specified form in time, with the fixed rank $S$. By now it is well known that the collection of functions of this form admits an infinite-dimensional manifold structure \cite[Section 3]{Falco.A_Hackbusch_Nouy_2019_FoCM}.
Besides the unbounded operator $\Lambda$, the resulting system of equations involves also a non-linear projection operator onto the tangent space to the manifold, which makes its analysis difficult and non-standard. 

Our strategy is to work with a suitable set of parameters describing the manifold, that are elements of a suitable ambient Hilbert space, and invoke results for
the evolutionary equations in linear spaces. In utilising such results,
the right choice of parametrisation turns out to be crucial. Our choice
of parameters leads us to the so-called Dual DO formulation introduced in  \cite{Musharbash.E_Nobile_2018_Dual}.

A method similar to the DLR approximation is the multi-configuration
time-dependent Hartree (MCTDH) method, which has been considered in
the context of computational quantum chemistry to approximate a deterministic
Schrödinger equation. For the MCTDH method, several existence results
have been established, e.g.\ \cite{Koch.O_Lubich_2007_regularity_existence,Bardos.C_etal_2010,Koch.O_Lubich_2011_IMA}.
The strategy used in these papers, first proposed by Koch and Lubich
\cite{Koch.O_Lubich_2007_regularity_existence}, is to consider a
constraint called the gauge condition that is defined by the differential
operator in the equation. 
With their choice of the gauge condition and their specific setting, the differential operator appears outside the projection operator, and this was a crucial step in \cite{Koch.O_Lubich_2007_regularity_existence,Bardos.C_etal_2010,Koch.O_Lubich_2011_IMA}
to apply the standard theory of abstract Cauchy problems. However, as
we will see later in Section~\ref{subsec:Discussion-on-gauge}, the
same approach does not work {in our setting}.

As mentioned above, our strategy in this paper is to work with the Dual DO formulation, by which we are able to show that the DLR approximation exists as long as a suitable full
rank condition is satisfied. Further, we discuss the extendability
of the approximation, beyond the point where we lose the full rankness. 

The rest of this paper is organised as follows. 
In Section~\ref{sec:setting}, we introduce the problem {under study}:   
the DLR equation and its equivalent formulation called Dual DO equation.
Section~\ref{sec:param-eq} introduces a parameter-equation that is equivalent to the Dual DO equations. 
Then, in Section~\ref{sec:Existence-and-Regularity} we prove our main result, namely the existence and uniqueness of a DLR solution on the maximal interval. 
The solution evolves in a manifold up to a maximal time. The solution cannot be continued in this manifold, but we will show that it can be extended in the ambient space, and the resulting continuation will take values in a different manifold with lower rank. 
Finally, in Section~\ref{sec:conclusions} we draw some conclusions.

\section{DLR formulation}\label{sec:setting}
In this section, we introduce the setting and recall some facts on
the Dynamical Low Rank (DLR) approach that will be needed later. 

We detail in Section~\ref{sec:Assumptions} the precise assumptions on $\Lambda$, $F$ and the initial conditions we will work with. 
For the moment, we just assume that a solution of \eqref{eq:exact-eq} exists.
We note, however, that the existence and uniqueness can be established by standard arguments.
For instance, if $\Lambda$ is self-adjoint and satisfies $\langle-\Lambda x,x\rangle\geq0$
for all $x\in D_{\mathcal{H}}(\Lambda)$,
by extending the definition of $\Lambda$ to random functions $u\in L^{2}(\Omega;\mathcal{H})$, where $\Lambda\colon 
D(\Lambda)\subset L^{2}(\Omega;\mathcal{H})
\to L^{2}(\Omega;\mathcal{H})$ is applied pointwise in $\Omega$, 
we have that $\Lambda$ is densely defined, closed, and satisfies
\[
\mathbb{E}[\langle-\Lambda v,v\rangle]\geq0\quad\text{for all }v\in D(\Lambda)\subset L^{2}(\Omega;\mathcal{H}).
\]
Together with a local Lipschitz continuity of $F$, existence of solutions can be established by invoking a standard theory of semi-linear evolution equations, see for
example \cite{Pazy.A_1983_book,Sell.G_You_2013_book}.

The DLR approach seeks an approximate solution of the equation \eqref{eq:exact-eq}
defined by $S$ deterministic and $S$ random basis functions. To be more precise, we define an element $u_{S}\in L^{2}(\Omega;\mathcal{H})$
to be an $S$-\textit{rank random field} if $u_{S}$ can be expressed
as a linear combination of $S$ (and not less than $S$) linearly
independent elements of $\mathcal{H}$, and $S$ (and not less than
$S$) linearly independent elements of $L^{2}(\Omega)$. Further,
we let $\hat{M}_{S}\subset L^{2}(\Omega;\mathcal{H})$ be the collection
of all the $S$-rank random fields:
\begin{align*}
&\hat{M}_{S}\\
&\!:=\!\!\:\Bigg\{u_{S}\!=\!\!\!\:\sum_{j=1}^{S}U_{i}Y_{i}
\left|
\begin{array}{rl}
\mathrm{span}_{\mathbb{R}}\{\{U_{j}\}_{j=1}^{S}\}&\!\text{is an $S$ dimensional subspace of }\mathcal{H}\\[2.5pt]
\mathrm{span}_{\mathbb{R}}\{\{Y_{j}\}_{j=1}^{S}\}&\!\text{is an $S$ dimensional subspace of }L^{2}(\Omega)\!\!\end{array}\right.
\Bigg\}.
\end{align*}
It is known that $\hat{M}_{S}$ can be equipped with a differentiable
manifold structure, see \cite{Musharbash.E_Nobile_2018_Dual,Falco.A_Hackbusch_Nouy_2019_FoCM}
and references therein. The idea behind the DLR approach is to approximate
the curve $t\mapsto u(t)\in L^2(\Omega;\mathcal{H})$ defined by the solution of
the equation \eqref{eq:exact-eq} by a curve $t\mapsto u_{S}(t)\in\hat{M}_{S}$
given as a solution of the following problem: find $u_{S}\in\hat{M}_{S}$
such that $u_{S}(0)={u_{0S}}\in\hat{M}_{S}$, {a suitable approximation of $u_0$ in $\hat{M}_{S}$}, and for (almost) all $t>0$
{we have $\frac{\partial u_{S}}{\partial t}(t)-(\Lambda u_{S}(t)+F(u_{S}(t)))\in L^{2}(\Omega;\mathcal{H})$
	and}
\begin{equation}
\mathbb{E}\Big[\Big\langle\frac{\partial u_{S}}{\partial t}(t)-(\Lambda u_{S}(t)+F(u_{S}(t))),v\Big\rangle \Big]=0,\;\text{ for all }v\in T_{u_{S}(t)}\hat{M}_{S},\label{eq:prob-var}
\end{equation}
where $T_{u_{S}(t)}\hat{M}_{S}{\subset L^{2}(\Omega;\mathcal{H})}$
is the tangent space of $\hat{M}_{S}$ at $u_{S}(t)${, and $\mathbb{E}[\cdot]$ denotes expectation with respect to the underlying probability measure $\mathbb{P}$.}

In this paper, we search for the solution in the same set as $\hat{M}_{S}$
but with a different parametrisation that is easier to work with.
The set \begin{equation}
{M}_{S}:=\Bigg\{u_{S}\!=\!\sum_{j=1}^{S}U_{i}Y_{i}
\left|
\begin{array}{rl}
\{U_{j}\}_{j=1}^{S}&\!\!\text{is linear independent in }\mathcal{H}\\[2.5pt]
\{Y_{j}\}_{j=1}^{S}&\!\!\text{is orthonormal in }L^{2}(\Omega)\end{array}
\right\}
\label{eq:def-MS}
\end{equation} is the same subset of $L^{2}(\Omega;\mathcal{H})$ as $\hat{M}_{S}$,
and thus the above problem is equivalent when we seek solutions in
$M_{S}$ instead of $\hat{M}_{S}$. This leads us to the so-called
Dual Dynamically Orthogonal (DO) formulation of the problem \eqref{eq:prob-var}.

For $u_{S}=\boldsymbol{U}^{\top}\boldsymbol{Y}\in M_{S}$, we define
the operator $\mathscr{P}_{u_{S}}\colon L^{2}(\Omega;\mathcal{H})\to L^{2}(\Omega;\mathcal{H})$
by
\[
\mathscr{P}_{u_{S}}:=P_{\boldsymbol{U}}+P_{\boldsymbol{Y}}-P_{\boldsymbol{U}}P_{\boldsymbol{Y}},
\]
where, for an arbitrary $\mathcal{H}$-orthonormal basis $\{\phi_{j}\}_{j=1}^{S}\subset\mathcal{H}$
of $\mathrm{span}_{\mathbb{R}}\{\{U_{j}\}_{j=1}^{S}\}$ the operator
$P_{\boldsymbol{U}}\colon L^{2}(\Omega;\mathcal{H})\to L^{2}(\Omega;\mathcal{H})$
is defined by
\[
P_{\boldsymbol{U}}f=\sum_{j=1}^{S}\langle f,\phi_{j}\rangle\phi_{j}\ \text{for }f\in L^{2}(\Omega;\mathcal{H}),
\]
and moreover, for an arbitrary $L^{2}(\Omega)$-orthonormal basis $\{\psi_{j}\}_{j=1}^{S}\subset L^{2}(\Omega)$
of $\mathrm{span}_{\mathbb{R}}\{\{Y_{j}\}_{j=1}^{S}\}$ the operator
$P_{\boldsymbol{Y}}\colon L^{2}(\Omega;\mathcal{H})\to L^{2}(\Omega;\mathcal{H})$
is defined by
\begin{equation}
P_{\boldsymbol{Y}}f=\sum_{j=1}^{S}\mathbb{E}[f\psi_{j}]\psi_{j}\ \text{for }f\in L^{2}(\Omega;\mathcal{H}).\label{eq:def-PY}
\end{equation}
This operator $\mathscr{P}_{u_{S}}$ turns out to be the $L^{2}(\Omega;\mathcal{H})$-orthogonal
projection to the tangent space $T_{u_{S}}M_{S}$ at $u_{S}=\boldsymbol{U}{}^{\top}\boldsymbol{Y}$,
see \cite[Proposition 3.3]{Musharbash.E_etal_2015_SISC} together
with \cite{Conte.D_Lubich_2010_error_MCTDH}. We note that the operator $\mathscr{P}_{u_{S}}$ is independent
of the choice of the representation of $u_{S}$: for any full rank
matrix $C\in\mathbb{R}^{S\times S}$ we have $(C^{\top}\boldsymbol{U})^{\top}C^{-1}\boldsymbol{Y}=u_{S}=\boldsymbol{U}^{\top}\boldsymbol{Y}$,
but also $P_{[C^{\top}\boldsymbol{U}]}=P_{\boldsymbol{U}}$ and $P_{[C^{-1}\boldsymbol{Y}]}=P_{\boldsymbol{Y}}$.

Using the above definitions, the problem we consider, equivalent to
\eqref{eq:prob-var}, can be formulated as follows: 
\begin{problem*}
	Find ${t\mapsto u_{S}(t)}\in M_{S}$ such that $u_{S}(0)={u_{0S}}\in M_{S}$ and for
	$t>0$ we have
	\begin{equation}
	\frac{\partial u_{S}}{\partial t}(t)=\mathscr{P}_{u_{S}(t)}(\Lambda u_{S}(t)+F(u_{S}(t))).\label{eq:original-eq-with-proj}
	\end{equation}
	In this paper, we consider two notions of solutions of this problem:
	the strong and classical solution.
\end{problem*}
\begin{definition}
	\label{def:original-strong}A function $u_{S}\colon[0,T]\to M_{S}\subset L^{2}(\Omega;\mathcal{H})$
	is called a \textit{strong solution} of the initial value problem
	\eqref{eq:original-eq-with-proj} if $u_{S}(0)={u_{0S}}\in M_{S}$, $u_{S}$ is {absolutely continuous on $[0,T]$,}
	and \eqref{eq:original-eq-with-proj} is satisfied a.e.\ on $[0,T]$.
	Further, we call $u_S$ a strong solution on $[0,T)$ if it is a strong solution on any subinterval $[0,T']\subset[0,T)$.
\end{definition}
In practice, further regularity of $u_{S}$ may be of interest.
\begin{definition}
	\label{def:original-classical}A function $u_{S}\colon[0,T]\to M_{S}\subset L^{2}(\Omega;\mathcal{H})$
	is called a \textit{classical solution} of \eqref{eq:original-eq-with-proj}
	on $[0,T]$ if $u_{S}(0)={u_{0S}}\in M_{S}$, $u_{S}$ is absolutely continuous on $[0,T]$, continuously differentiable on {$(0,T]$}, $u_{S}{\in D(\Lambda)}$
	for $t\in{(0,T]}$, and \eqref{eq:original-eq-with-proj} is satisfied
	on {$(0,T]$}.
	Further, we call $u_S$ a classical solution on $[0,T)$ when it is a classical solution on any subinterval $[0,T']\subset[0,T)$.
\end{definition}

\subsection{Dual DO formulation}
Our aim is to establish the unique existence of a solution to problem \eqref{eq:original-eq-with-proj}.
A difficulty is that $u_{S}$ propagates in a non-linear manifold ${M}_{S}$. 
Our strategy is to choose a suitable parametrisation of ${M}_{S}$, and work in a linear space {which the parameters belong to}.
For the parametrisation, we will choose the one proposed in \cite{Musharbash.E_Nobile_2018_Dual}, which results in a formulation of  \eqref{eq:original-eq-with-proj} 
called Dual DO, 
where we seek an approximate solution of the form $u_{S}(t)=\boldsymbol{U}^{\top}(t)\boldsymbol{Y}(t)\in M_{S}$ for any $[0,T]$.
Here, the parameter $(\boldsymbol{U}(t),\boldsymbol{Y}(t))\in[\mathcal{H}]^{S}\times[L^{2}(\Omega)]^{S}$
is a solution to the following problem:
\begin{enumerate}
	\item the components of $\boldsymbol{U}(t)=(U_{1}(t),\dots,U_{S}(t))^{\top}$ are linearly
	independent in $\mathcal{H}$ for any $t\in[0,T]$;
	\item the components of $\boldsymbol{Y}(t)=(Y_{1}(t),\dots,Y_{S}(t))^{\top}$ are orthonormal
	in $L^{2}(\Omega)$, and satisfy the so-called gauge condition: for any $t\in(0,T)$,
	\[
	\mathbb{E}\left[\frac{\partial Y_{j}}{\partial t}Y_{k}\right]=0\ \text{for}\ j,k=1,\dotsc,S,\text{ equivalently, }\mathbb{E}\left[\frac{\partial\boldsymbol{Y}}{\partial t}\boldsymbol{Y}^{\top}\right]=0\in\mathbb{R}^{S\times S};
	\]
	\item $(\boldsymbol{U},\boldsymbol{Y})$ satisfies the equation
	\begin{equation}
	\left\{ \begin{array}{rl}
	\frac{\partial}{\partial t}\boldsymbol{U}& =\mathbb{E}\left[\mathcal{L}(u_S)\boldsymbol{Y}\right]\\
	Z_{\boldsymbol{U}}\frac{\partial}{\partial t}\boldsymbol{Y} & =(I-P_{\boldsymbol{Y}})\left[\langle\mathcal{L}(u_S),\boldsymbol{U}\rangle\right],
	\end{array}\right.\label{eq:general-dualDO}
	\end{equation}
	where $\mathcal{L}:={\Lambda+F}$,
	$P_{\boldsymbol{Y}}$ is as in \eqref{eq:def-PY}, and $Z_{\boldsymbol{U}}={(\langle U_{j},U_{k}\rangle)_{j,k=1,\dotsc,S}}\in\mathbb{R}^{S\times S}$
	is the Gram matrix defined by $\boldsymbol{U}$;
	\item $(\boldsymbol{U},\boldsymbol{Y})$ satisfies the initial condition $(\boldsymbol{U}(0),\boldsymbol{Y}(0))=(\boldsymbol{U}_{0},\boldsymbol{Y}_{0})$ for some $(\boldsymbol{U}_{0},\boldsymbol{Y}_{0})\in[\mathcal{H}]^{S}\times[L^{2}(\Omega)]^{S}$
	such that 
	$\boldsymbol{U}_{0}^{\top}\boldsymbol{Y}_{0}{=u_{0S}}\in M_{S}$.
\end{enumerate}
Noting that, since the operator $\Lambda$ is deterministic and linear, we have \[P_{\boldsymbol{Y}}(\langle\Lambda(u_{S}),\boldsymbol{U}\rangle)=P_{\boldsymbol{Y}}(\langle\Lambda(\boldsymbol{U}^{\top})\boldsymbol{Y},\boldsymbol{U}\rangle)=\langle\Lambda(u_{S}),\boldsymbol{U}\rangle
\]
and 
$\mathbb{E}[\Lambda(u_{S})\boldsymbol{Y}^{\top}]
=\Lambda(\boldsymbol{U}^{{\top}})
\mathbb{E}[
\boldsymbol{Y}\boldsymbol{Y}^{\top}
]
=\Lambda(\boldsymbol{U}^\top)$,
the equation \eqref{eq:general-dualDO} reads 
\begin{alignat}{1}
\left\{\begin{array}{rll}
\frac{\partial}{\partial t}\boldsymbol{U} & =\Lambda(\boldsymbol{U})+\mathbb{E}\left[F(\boldsymbol{U}^{\top}\boldsymbol{Y})\boldsymbol{Y}\right] & =:\Lambda(\boldsymbol{U})+G_{1}(\boldsymbol{Y})(\boldsymbol{U})\\
\frac{\partial}{\partial t}\boldsymbol{Y} & =(I-P_{\boldsymbol{Y}})(\langle F(\boldsymbol{U}^{\top}\boldsymbol{Y}),Z_{\boldsymbol{U}}^{-1}\boldsymbol{U}\rangle) & =:G_{2}(\boldsymbol{U})(\boldsymbol{Y}).
\end{array}\right.\label{eq:working-eq}
\end{alignat}
%\footnote{For any vector $\boldsymbol{z}_{j}\in\mathbb{R}^{S\times1}$, from the linearity of $P_{\boldsymbol{Y}}$ and $\langle\cdot,v\rangle$ we have 
%\[
%\boldsymbol{z}_{j}^{\top}[P_{\boldsymbol{Y}}(\langle\boldsymbol{U},v\rangle)]=P_{\boldsymbol{Y}}(\boldsymbol{z}_{j}^{\top}\langle\boldsymbol{U},v\rangle)=P_{\boldsymbol{Y}}(\langle\boldsymbol{z}_{j}^{\top}\boldsymbol{U},v\rangle)\in\mathbb{R}.
%\]
%Thus, for $Z^{-1}:=(\boldsymbol{z}_{1}|\cdots|\boldsymbol{z}_{S})^{\top}$, we have
%\begin{align*}
%Z^{-1}[P_{\boldsymbol{Y}}(\langle\boldsymbol{U},v\rangle)] & =\left(\boldsymbol{z}_{1}^{\top}P_{\boldsymbol{Y}}(\langle\boldsymbol{U},v\rangle),\dots,\boldsymbol{z}_{S}^{\top}P_{\boldsymbol{Y}}(\langle\boldsymbol{U},v\rangle)\right)^{\top}\\
%& %=\left(P_{\boldsymbol{Y}}(\langle\boldsymbol{z}_{1}^{\top}\boldsymbol{U},v\rangle),\dots,P_{\boldsymbol{Y}}(\langle\boldsymbol{z}_{S}^{\top}\boldsymbol{U},v\rangle)\right)^{\top}=P_{\boldsymbol{Y}}(\langle Z^{-1}\boldsymbol{U},v\rangle)\in\mathbb{R}^{S\times1}.
%\end{align*}
%}
We define two notions of solutions to the initial value problem of
\eqref{eq:working-eq} that correspond to those of the original problem
as in Definitions~\ref{def:original-strong}--\ref{def:original-classical}.
\begin{definition}
	\label{def:Dual-DO-strong}A function $(\boldsymbol{U},\boldsymbol{Y})\colon[0,T]\to[\mathcal{H}]^{S}\times[L^{2}(\Omega)]^{S}$
	is called a \textit{Dual DO solution of the problem
		\eqref{eq:original-eq-with-proj} in the strong sense} if $(\boldsymbol{U},\boldsymbol{Y})$
	satisfies the following conditions:
	\begin{enumerate}
		\item $(\boldsymbol{U}(0),\boldsymbol{Y}(0))=(\boldsymbol{U}_{0},\boldsymbol{Y}_{0})$
		for some $(\boldsymbol{U}_{0},\boldsymbol{Y}_{0})\in[\mathcal{H}]^{S}\times[L^{2}(\Omega)]^{S}$
		such that ${u_{0S}}=\boldsymbol{U}_{0}^{\top}\boldsymbol{Y}_{0}\in M_{S}$;
		\item $(\boldsymbol{U},\boldsymbol{Y})$ satisfies the equation \eqref{eq:working-eq}
		a.e.\ on $[0,T]$;
		\item the curve $t\mapsto\boldsymbol{U}(t)\in[\mathcal{H}]^{S}$ is {absolutely continuous on $[0,T]$};
		\item the curve $t\mapsto\boldsymbol{Y}(t)\in[L^{2}(\Omega)]^{S}$ is {absolutely continuous on $[0,T]$};
		\item \label{enu:U-indep}$\{U_{j}(t)\}_{j=1}^{S}$ is linear independent
		in $\mathcal{H}$ for almost every $t\in[0,T]$; and
		\item $\{Y_{j}(t)\}_{j=1}^{S}$ is orthonormal in $L^{2}(\Omega)$ for almost
		every $t\in[0,T]$.
	\end{enumerate}
	Notice, in particular, that the condition~\ref{enu:U-indep} above
	implies that the matrix $Z_{\boldsymbol{U}}$ is invertible for almost
	every $t\in[0,T]$. Further, from \eqref{eq:working-eq} we necessarily
	have
	\begin{equation}
	\mathbb{E}\Big[\Big(\frac{\partial}{\partial t}\boldsymbol{Y}\Big)\boldsymbol{Y}^{\top}\Big]=\mathbb{E}\Big[\langle F(\boldsymbol{U}^{\top}\boldsymbol{Y}),Z_{\boldsymbol{U}}^{-1}\boldsymbol{U}\rangle(I-P_{\boldsymbol{Y}})\boldsymbol{Y}^{\top}\Big]=0.\label{eq:implied-gauge}
	\end{equation}
\end{definition}

\begin{definition}
	\label{def:Dual-DO-classical}A function $(\boldsymbol{U},\boldsymbol{Y})\colon[0,T]\to[\mathcal{H}]^{S}\times[L^{2}(\Omega)]^{S}$
	is called a \textit{{Dual DO solution of the problem
			\eqref{eq:original-eq-with-proj}} in the classical sense} if $(\boldsymbol{U},\boldsymbol{Y})$
	satisfies the following conditions:
	\begin{enumerate}
		\item $(\boldsymbol{U}(0),\boldsymbol{Y}(0))=(\boldsymbol{U}_{0},\boldsymbol{Y}_{0})$
		for some $(\boldsymbol{U}_{0},\boldsymbol{Y}_{0})\in[\mathcal{H}]^{S}\times[L^{2}(\Omega)]^{S}$
		such that ${u_{0S}}=\boldsymbol{U}_{0}^{\top}\boldsymbol{Y}_{0}\in M_{S}$;
		\item $(\boldsymbol{U},\boldsymbol{Y})$ satisfies the equation \eqref{eq:working-eq}
		on ${(0,T]}$;
		\item the curve $t\mapsto\boldsymbol{U}(t)\in[\mathcal{H}]^{S}$ is 
		{absolutely continuous on $[0,T]$, continuously differentiable on $(0,T]$};
		\item the curve $t\mapsto\boldsymbol{Y}(t)\in[L^{2}(\Omega)]^{S}$ is
		{absolutely continuous on $[0,T]$, continuously differentiable on $(0,T]$};
		\item $U_{j}(t)\in D_{\mathcal{H}}(\Lambda)$ for any $t\in{(0,T]}$, $j=1,\dots,S$;
		\item $\{U_{j}(t)\}_{j=1}^{S}$ is linear independent in $\mathcal{H}$
		for any $t\in[0,T]$;
		\item $\{Y_{j}(t)\}_{j=1}^{S}$ is orthonormal in $L^{2}(\Omega)$ for any
		$t\in[0,T]$.
	\end{enumerate}
\end{definition}

\begin{definition}
	If $(\boldsymbol{U},\boldsymbol{Y})\colon[0,T)\to[\mathcal{H}]^{S}\times[L^{2}(\Omega)]^{S}$
	is a Dual DO solution on all subintervals $[0,T']\subset[0,T)$
	in the strong (resp.~classical) sense, then we call $(\boldsymbol{U},\boldsymbol{Y})$
	a Dual DO solution on $[0,T)$ in the strong (resp.~classical)
	sense.
\end{definition}

As we will see in the next section, establishing the unique existence
of the Dual DO solution is equivalent to establishing the unique existence of solutions to
the original equation \eqref{eq:original-eq-with-proj}. Thus, for
the rest of this paper we will work with the Dual DO formulation. 
\subsection{\label{subsec:equiv-DLR-DO}Equivalence with the original formulation}
In this section, we establish the equivalence of the original
equation \eqref{eq:original-eq-with-proj} and the Dual DO {formulation}
as in Definitions~\ref{def:Dual-DO-strong}--\ref{def:Dual-DO-classical}. 
Our first step is to show that if a solution $u_{S}$ of the original
equation \eqref{eq:original-eq-with-proj} is given, then there exists
a unique solution of \eqref{eq:working-eq} that is also the unique Dual DO solution $t\mapsto(\boldsymbol{U}(t),\boldsymbol{Y}(t))\in[\mathcal{H}]^{S}\times[L^{2}(\Omega)]^{S}$ of \eqref{eq:original-eq-with-proj} 
such that $u_{S}=\boldsymbol{U}^{\top}\boldsymbol{Y}$, see Lemma \ref{lem:unique-lift}.

We will need a proposition which states that 
if $t\mapsto u_S(t)\in M_S\subset L^{2}(\Omega;\mathcal{H})$ is differentiable, then there exists a differentiable parametrisation. 
This result may be seen as a generalisation of the existence of smooth singular value decompositions of matrix-valued curve considered, for example, in \cite{Dieci.L_Eirola_1999_SIMAX,Chern.J_Dieci_2000_SIMAX_smooth_decomp}.
We start with the following lemma, which shows the existence of the singular value decomposition for elements in $M_S$. 
\begin{lemma}
	\label{lem:uS-SVD}Let $u_{S}\in M_{S}\subset L^{2}(\Omega;\mathcal{H})$
	be given. Then, with some $\{\tilde{V}_{j}\}_{j=1}^{S}$ and $\{W_{j}\}_{j=1}^{S}$
	orthonormal in $\mathcal{H}$ and $L^{2}(\Omega)$, respectively,
	and $\sigma_{j}>0$, $j=1,\dots,S$, we have 
	\[
	u_{S}=\sum_{j=1}^{S}\sigma_{j}\tilde{V}_{j}W_{j}.
	\]
	Moreover, such $\sigma_{j}>0$ is unique in the following sense: for any other representation
	$u_{S}=\sum_{j=1}^{S}\sigma'_{j}\tilde{V}'_{j}W'_{j}$ with $\{\tilde{V}_{j}'\}_{j=1}^{S}$
	and $\{W_{j}'\}_{j=1}^{S}$ orthonormal, upon relabelling if necessary,
	we have $\sigma'_{j}=\sigma{}_{j}$, $j=1,\dots,S$. Furthermore,
	if $[0,T]\ni t\mapsto u_{S}(t)\in M_{S}\subset L^{2}(\Omega;\mathcal{H})$
	is continuous, then the corresponding values $\{\sigma_{j}(t)\}_{j=1}^{S}$
	satisfy 
	\begin{equation}
	0<\min_{j=1,\dots,S}\inf_{t\in[0,T]}\sigma_{j}(t)\qquad\text{and}\qquad\max_{j=1,\dots,S}\sup_{t\in[0,T]}\sigma_{j}(t)<\infty.\label{eq:bds-eig}
	\end{equation}
	
\end{lemma}

\begin{proof}
	{The linear operator $K=K(u_{S})$ defined by $L^{2}(\Omega)\ni w\mapsto Kw:=\mathbb{E}[u_{S}w]\in\mathcal{H}$
		is a finite-rank operator with rank $S$, with the image being independent
		of the representation of $u_{S}=\boldsymbol{U}^{\top}\boldsymbol{Y}\in M_{S}$.}
	%\footnote{{Suppose, to derive a contradiction, that $u_{S}=\boldsymbol{U}^{\top}\boldsymbol{Y}=\boldsymbol{V}^{\top}\boldsymbol{W}$
	%and $\mathrm{span}\{U_{j}\mid j=1,\dots,S\}\neq\mathrm{span}\{V_{j}\mid j=1,\dots,S\}$.
	%Then, since both are $S$-dimensional space, neither can be a proper subset of the other. Hence, for some $k$ we have $U_{k}\not\in\mathrm{span}\{V_{j}\mid j=1,\dots,S\}$.
	%But multiplying $Y_{k}$ to the both sides of $\boldsymbol{U}^{\top}\boldsymbol{Y}=\boldsymbol{V}^{\top}\boldsymbol{W}$
	%and taking expectation yields
	%\[U_{k}=\sum_{j=1}^{S}c_{j}V_{j}\quad\text{with }c_{j}:=\mathbb{E}[W_{j}Y_{k}],\]
	%which contradicts $U_{k}\not\in\mathrm{span}\{V_{j}\mid j=1,\dots,S\}$.}}
	%
	Thus, 
	with some $\{\tilde{V}_{j}\}_{j=1}^{S}$ and $\{W_{j}\}_{j=1}^{S}$ orthonormal in {$\mathcal{H}$ and $L^{2}(\Omega)$},  respectively, $K$ admits the canonical decomposition
	
	\[
	Kw=\sum_{j=1}^{S}\sigma_{j}\mathbb{E}[wW_{j}]\tilde{V}_{j},
	\]
	with singular values $\sigma_{j}=\sigma_{j}(K)>0$, $j=1,\dots,S$,
	see e.g.~\cite[Sections III.4.3 and V.2.3]{Kato.T_book_1995_reprint}.
	If we have another representation $u_{S}=\sum_{j=1}^{S}\sigma'_{j}\tilde{V}'_{j}W'_{j}$,
	then upon relaballing if necessary, we must have $\sigma'_{j}=\sigma{}_{j}$.
	To see this, first note that the adjoint operator $\mathcal{H}\ni v\mapsto K^{*}v:=\langle u_{S},v\rangle\in L^{2}(\Omega)$
	is a finite-rank operator with rank $S$. The operator $K^{*}K$ is
	also rank $S$ and admits the spectral decomposition
	\[
	K^{*}Kw=\sum_{j=1}^{S}\sigma_{j}^{2}\mathbb{E}[wW_{j}]W_{j},
	\]
	with eigenvalues $\{\sigma_{j}^{2}\}_{j=1}^{S}$ and the corresponding
	eigenfuncitons $\{W_{j}\}_{j=1}^{S}$. Similarly, if we have a representation
	$u_{S}=\sum_{j=1}^{S}\sigma'_{j}\tilde{V}'_{j}W'_{j}$, then $\{W'_{k}\}_{k=1}^{S}$
	are also eigenfunctions of $K^{*}K$ corresponding to the eigenvalues
	$\{(\sigma_{k}')^{2}\}_{k=1}^{S}$. Thus, for the image of $K^{*}K$
	to be $S$-dimesnional, we must have $\{\sigma_{k}'\mid k=1,\dots,S\}=\{\sigma_{j}\mid j=1,\dots,S\}$,
	and moreover each eigenvalue $\sigma_{j}'=\sigma_{k}$ must have the
	same (geometric) multiplicity.
	
	To show \eqref{eq:bds-eig}, relabel $\{\sigma_{j}(t)\}_{j=1}^{S}$
	in the non-decreasing order and denote it by $(\alpha_{j}(t))_{j=1}^{S}$.
	Then, for any $t\in[0,T]$ and  $h\in\mathbb{R}$ such that $t+h\in[0,T]$ we have 
	\begin{align*}
	|\alpha_{j}(t+h)-\alpha_{j}(t)| & \leq\|K(u_{S}(t+h))-K(u_{S}(t))\|_{L^{2}(\Omega)\to\mathcal{H}}\quad\text{for }j=1,\dots,S,
	\end{align*}
	see for example \cite[Proposition II.7.6 and Theorem IV.2.2]{Pinkus.A_1985_book_n_width}.
	But for any $w\in L^{2}(\Omega)$ we have 
	\begin{align*}
	\|K(u_{S}(t+h))w-K(u_{S}(t))w\|_{\mathcal{H}} % &=\|\mathbb{E}[(u_{S}(t+h)-u_{S}(t))w]\|_{\mathcal{H}}\\
	%& \leq\mathbb{E}[\|u_{S}(t+h)-u_{S}(t)\|_{\mathcal{H}}w]\\
	& \leq\Big(\mathbb{E}[\|u_{S}(t+h)-u_{S}(t)\|_{\mathcal{H}}^{2}]\Big)^{1/2}\|w\|_{L^{2}(\Omega)},
	\end{align*}
	and thus the continuity of $t\mapsto u_{S}(t)$ implies that $\alpha_{j}$
	is continuous on $[0,T]$. Now, since $K$ is of rank $S$, we have
	$\alpha_{j}(t)>0$ for any $t\in[0,T]$. Hence, for any $j=1,\dots,S$
	we have 
	\[
	\inf_{t\in[0,T]}\sigma_{j}(t)\geq\min_{t\in[0,T]}\alpha_{1}(t)>0.
	\]
	Similarly, we have $\sup_{t\in[0,T]}\sigma_{j}(t)\leq\max_{t\in[0,T]}\alpha_{S}(t)<\infty$,
	which completes the proof.
	\end{proof}
\begin{proposition}\label{prop:smooth-SVD}
	Suppose that $[0,T]\ni t\mapsto u_{S}(t)\in M_{S}\subset L^{2}(\Omega;\mathcal{H})$
	is absolutely continuous. Then, there exist $t\mapsto\tilde{V}_{j}(t)\in\mathcal{H}$,
	$t\mapsto\Sigma(t)\in\mathbb{R}^{S\times S}$, and 
	$t\mapsto W_{j}(t)\in L^{2}(\Omega)$,
	$j=1,\dots,S$ such that 
	\[
	u_{S}(t)=
	\tilde{\boldsymbol{V}}(t)^\top\Sigma(t)\boldsymbol{W}(t)
	\qquad\text{for all }t\in[0,T];
	\]
	$\{\tilde{V}_{j}(t)\}_{j=1}^{S}$ and $\{W_{j}(t)\}_{j=1}^{S}$ are
	orthonormal in $\mathcal{H}$ and in $L^{2}(\Omega)$, respectively; $\Sigma(t)$ is full rank; 
	the curves 
	$t\mapsto\Sigma(t)\in\mathbb{R}^{S\times S}$,  $t\mapsto\tilde{V}_{j}(t)\in\mathcal{H}$, 
	and $t\mapsto W_{j}(t)\in L^{2}(\Omega)$, $j=1,\dots,S$ are absolutely
	continuous on $[0,T]$. 
	Moreover, if $u_{S}(t)$ is continuously differentiable on $(0,T]$, then 
	$\tilde{V}_{j}(t)$, $\Sigma(t)$, and $W_{j}(t)$ are continuously differentiable on $(0,T]$. 
	In particular, $u_{S}(t)$ admits a representation
	$u_{S}(t)=\boldsymbol{V}(t)^{\top}\boldsymbol{W}(t)$ in $M_{S}$ with $\boldsymbol{V}^{\top}=\tilde{\boldsymbol{V}}^{\top}\Sigma$, 
	with the specified smoothness.
\end{proposition}
To show Proposition~\ref{prop:smooth-SVD}, we will use an argument similar to what we will see in Section~\ref{sec:Existence-and-Regularity} below. Thus, we will defer the proof to Section~\ref{sec:Existence-and-Regularity}.

Parametrisation of $M_{S}$ is determined by parameters up to a unique orthogonal
matrix.
\begin{lemma}
	\label{lem:up-to-O(S)}Let $v_{S}\in M_{S}$ be given. Suppose that
	$v_{S}$ admits two representations $v_{S}=\boldsymbol{V}^{\top}\boldsymbol{W}=\tilde{\boldsymbol{V}}^{\top}\tilde{\boldsymbol{W}}\in M_{S}$
	with some $(\boldsymbol{V},\boldsymbol{W}),\ (\tilde{\boldsymbol{V}},\tilde{\boldsymbol{W}})\in[H]^{S}\times[L^{2}(\Omega)]^{S}$
	satisfying the linear independence and orthonormality {conditions}
	as in \eqref{eq:def-MS}. Then, we have
	\[
	(\tilde{\boldsymbol{V}},\tilde{\boldsymbol{W}})=(\Theta{}^{\top}\boldsymbol{V},\Theta{}^{\top}\boldsymbol{W}),
	\]
	for a unique $\Theta\in O(S)$. 
\end{lemma}
\begin{proof}
	From $\tilde{\boldsymbol{V}}^{\top}\tilde{\boldsymbol{W}}=\boldsymbol{V}{}^{\top}\boldsymbol{W}$,
	we have
	\[
	\tilde{\boldsymbol{W}}=(\langle\tilde{\boldsymbol{V}},\tilde{\boldsymbol{V}}{}^{\top}\rangle)^{-1}\langle\tilde{\boldsymbol{V}},\boldsymbol{V}{}^{\top}\rangle\boldsymbol{W}=:\Theta{}^{\top}\boldsymbol{W},
	\]
	so that $\tilde{\boldsymbol{W}}\tilde{\boldsymbol{W}}{}^{\top}=\Theta{}^{\top}\boldsymbol{W}\boldsymbol{W}{}^{\top}\Theta$.
	From the $L^{2}(\Omega)$-orthonormality of $\tilde{\boldsymbol{W}}$
	and $\boldsymbol{W}$, taking the expectation of both sides we conclude
	that $\Theta$ is an orthogonal matrix. To see the uniqueness, suppose
	\[
	\tilde{\boldsymbol{W}}=\tilde{\Theta}^{\top}\boldsymbol{W}\text{ for some }\tilde{\Theta}\in\mathbb{R}^{S\times S}.
	\]
	But from $\Theta{}^{\top}\boldsymbol{W}=\tilde{\Theta}^{\top}\boldsymbol{W}$
	and $\mathbb{E}[\boldsymbol{W}\boldsymbol{W}{}^{\top}]=I$, we must have
	$\tilde{\Theta}=\Theta$. 
	\end{proof}
The above lemma implies the following corollary, which states that
if both a solution $u_{S}$ of the original problem \eqref{eq:original-eq-with-proj}
and a Dual DO solution $(\boldsymbol{U}(t),\boldsymbol{Y}(t))$ of {\eqref{eq:original-eq-with-proj}}
exist, and if further the solution of the original problem is unique,
then $(\boldsymbol{U}(t),\boldsymbol{Y}(t))$ is determined by $u_{S}$
up to a unique orthogonal matrix. {We stress that the following corollary does not guarantee the uniqueness of the Dual DO solution. }
\begin{corollary}
	\label{cor:dual-DO-form} Suppose that the equation \eqref{eq:original-eq-with-proj}
	has a unique strong solution $u_{S}(t)\in M_{S}$, $t\in[0,T]$.
	{Let $(\boldsymbol{V}(t),\boldsymbol{W}(t))\in[H]^{S}\times[L^{2}(\Omega)]^{S}$ be any representation of $u_S(t)$, namely $u_{S}(t)=\boldsymbol{V}(t)^{\top}\boldsymbol{W}(t)$, }
	satisfying the linear independence and orthonormality conditions
	defined in \eqref{eq:def-MS}. 
	Furthermore, suppose that a Dual
	DO solution $(\boldsymbol{U}(t),\boldsymbol{Y}(t))$ exists in
	the strong sense. Then, we have
	\begin{equation}
	(\boldsymbol{U}(t),\boldsymbol{Y}(t))=(\Theta(t)^{\top}\boldsymbol{V}(t),\Theta(t)^{\top}\boldsymbol{W}(t)),\label{eq:dual-DO-form}
	\end{equation}
	for a unique $\Theta(t)\in O(S)$. In words, if a Dual DO solution
	$(\boldsymbol{U}(t),\boldsymbol{Y}(t))$ exists, then it must be of
	the form $(\Theta(t)^{\top}\boldsymbol{V}(t),\Theta(t)^{\top}\boldsymbol{W}(t))$
	with an arbitrarily chosen representation $\boldsymbol{V}(t)^{\top}\boldsymbol{W}(t)$
	of $u_{S}(t)$ and the corresponding unique orthogonal matrix $\Theta(t)$.
\end{corollary}
\begin{proof}
	{We first show that the function $\hat{u}_{S}:=\boldsymbol{U}(t)^{\top}\boldsymbol{Y}(t)\in M_{S}$
		satisfies the original equation \eqref{eq:original-eq-with-proj}.
		Since $(\boldsymbol{U}(t),\boldsymbol{Y}(t))$ is a Dual DO solution
		in the strong sense, from \eqref{eq:working-eq} a.e.\ on $[0,T]$
		we have
		\begin{align*}
		\frac{\mathrm{d}}{\mathrm{d}t}\hat{u}_{S} & =\frac{\mathrm{d}}{\mathrm{d}t}\boldsymbol{U}^{\top}\boldsymbol{Y}+\boldsymbol{U}^{\top}\frac{\mathrm{d}}{\mathrm{d}t}\boldsymbol{Y}\\
		& =\Lambda(\hat{u}_{S})+\boldsymbol{Y}^{\top}\mathbb{E}\left[F(\hat{u}_{S})\boldsymbol{Y}\right]+(I-P_{\boldsymbol{Y}})(\boldsymbol{U}^{\top}Z_{\boldsymbol{U}}^{-1}\langle F(\hat{u}_{S}),\boldsymbol{U}\rangle)\\
		& =\Lambda(\hat{u}_{S})+P_{\boldsymbol{Y}}(F(\hat{u}_{S}))+(I-P_{\boldsymbol{Y}})P_{\boldsymbol{U}}(F(\hat{u}_{S}))\in L^{2}(\Omega;\mathcal{H}).
		\end{align*}
		Now, notice that $P_{\boldsymbol{Y}}\Lambda(\hat{u}_{S})=\Lambda(\hat{u}_{S})$
		and thus $(P_{\boldsymbol{U}}-P_{\boldsymbol{U}}P_{\boldsymbol{Y}})\Lambda(\hat{u}_{S})=0$.
		Together with $P_{\boldsymbol{U}}P_{\boldsymbol{Y}}=P_{\boldsymbol{Y}}P_{\boldsymbol{U}}$
		we obtain
		\[
		\frac{\mathrm{d}}{\mathrm{d}t}\hat{u}_{S}=(P_{\boldsymbol{Y}}+(P_{\boldsymbol{U}}-P_{\boldsymbol{U}}P_{\boldsymbol{Y}}))\Lambda(\hat{u}_{S})+(P_{\boldsymbol{Y}}+P_{\boldsymbol{U}}-P_{\boldsymbol{U}}P_{\boldsymbol{Y}})F(\hat{u}_{S}),
		\]
		which is \eqref{eq:original-eq-with-proj}.}
	
	Then, from the uniqueness of the solution of the original problem
	we have $\boldsymbol{V}(t)^{\top}\boldsymbol{W}(t)=u_{S}=\boldsymbol{U}(t)^{\top}\boldsymbol{Y}(t)$.
	Thus, Lemma~\ref{lem:up-to-O(S)} implies \eqref{eq:dual-DO-form}, as claimed.
	\end{proof}
In the above corollary, we assumed the existence of both the solution
of the original problem and the Dual DO {formulation}, and deduced
the existence of a unique orthogonal matrix. The following lemma shows
that such an orthogonal matrix exists, showing that the unique existence
of the solution of the original problem {\eqref{eq:original-eq-with-proj}}
implies that of the Dual DO formulation as in Definitions~\ref{def:Dual-DO-strong}--\ref{def:Dual-DO-classical}.
The proof is inspired by \cite[Proof of Proposition II.3.1]{Kobayashi.S_Nomizu_1996_I}.
We will use the following lemma to show the equivalence of the original
problem \eqref{eq:original-eq-with-proj} and the Dual DO formulation
\eqref{eq:working-eq}, see Proposition~\ref{prop:existence-equiv}
below.
\begin{lemma}
	\label{lem:unique-lift}
	Suppose that $[0,T]\ni t\mapsto u_{S}(t)\in M_{S}\subset L^{2}(\Omega;\mathcal{H})$
	is absolutely continuous, $u_S(0)=u_{0S}\in M_S$, and satisfies the equation \eqref{eq:original-eq-with-proj} a.e.\ on $[0,T]$. 
	Let $(\boldsymbol{V}(0),\boldsymbol{W}(0))\in[\mathcal{H}]^{S}\times[L^{2}(\Omega)]^{S}$ be such that $\boldsymbol{V}(0)^\top \boldsymbol{W}(0)=u_{0S}$.  
	Then, there exists a Dual DO solution $(\boldsymbol{U}(t),\boldsymbol{Y}(t))\in[\mathcal{H}]^{S}\times[L^{2}(\Omega)]^{S}$
	in the strong sense with the initial condition $(\boldsymbol{V}(0),\boldsymbol{W}(0))\in[\mathcal{H}]^{S}\times[L^{2}(\Omega)]^{S}$.
	Further, $(\boldsymbol{U}(t),\boldsymbol{Y}(t))$ is the unique Dual
	DO solution such that $u_{S}(t)=\boldsymbol{U}(t)^{\top}\boldsymbol{Y}(t)$
	for all $t\in[0,T]$.
\end{lemma}
\begin{proof}
	From Proposition~\ref{prop:smooth-SVD}, there exists a curve $t\mapsto {(\tilde{\boldsymbol{V}}(t),\tilde{\boldsymbol{W}}(t))}\in[\mathcal{H}]^{S}\times[L^{2}(\Omega)]^{S}$
	such that $u_{S}(t)={\tilde{\boldsymbol{V}}(t)^{\top}\tilde{\boldsymbol{W}}(t)}$ for all $t\in[0,T]$; 
	$\{{\tilde{V}_{j}}\}_{j=1}^{S}$ is linear independent in $\mathcal{H}$;
	$\{{\tilde{W}_{j}}\}_{j=1}^{S}$ is orthonormal in $L^{2}(\Omega)$;  $t\mapsto{\tilde{\boldsymbol{V}}}(t)\in[\mathcal{H}]^{S}$ is absolutely continuous on $[0,T]$; 
	and 
	$t\mapsto{\tilde{\boldsymbol{W}}}(t)\in L^{2}(\Omega)$ is absolutely continuous on $[0,T]$.
	{In general, $\tilde{\boldsymbol{V}}(0)\neq\boldsymbol{V}(0)$ and $\tilde{\boldsymbol{W}}(0)\neq\boldsymbol{W}(0)$, but from Lemma~\ref{lem:up-to-O(S)}, one can find a unique orthogonal matrix ${\Xi}$ such that 
		\[
		\Xi\tilde{\boldsymbol{V}}(0)=\boldsymbol{V}(0)\ \text{ and }\ 
		\Xi\tilde{\boldsymbol{W}}(0)=\boldsymbol{W}(0).
		\]
		Now, let  $\Xi\tilde{\boldsymbol{V}}(t):=\boldsymbol{V}(t)$ and 
		$\Xi\tilde{\boldsymbol{W}}(t):=\boldsymbol{W}(t)$, so that $u_S(t)=\boldsymbol{V}^\top(t)\boldsymbol{W}(t)$. Notice that $t\mapsto\boldsymbol{V}$ and $t\mapsto\boldsymbol{W}$ are absolutely continuous. 
	}
	From Corollary~\ref{cor:dual-DO-form}, if the Dual DO solution
	$(\boldsymbol{U}(t),\boldsymbol{Y}(t))$ exists then we necessarily
	have
	\begin{equation}
	(\boldsymbol{U}(t),\boldsymbol{Y}(t))=(\Theta(t)^{\top}\boldsymbol{V}(t),\Theta(t)^{\top}\boldsymbol{W}(t)),\ \text{for some unique }\Theta(t)\in O(S).\label{eq:DO-form-in-pf}
	\end{equation}
	We show that such $\Theta(t)$, i.e.\ an orthogonal matrix $\Theta(t)$
	for which the pair $(\Theta(t)^{\top}\boldsymbol{V}(t),\Theta(t)^{\top}\boldsymbol{W}(t))$
	is a Dual DO solution, uniquely exists. Note that again from
	Corollary~\ref{cor:dual-DO-form}, it suffices to consider an arbitrarily
	fixed representation $(\boldsymbol{V}(t),\boldsymbol{W}(t))$.
	
	We will obtain $\Theta$ as a solution of an ordinary differential
	equation we will now derive. If $(\boldsymbol{U}(t),\boldsymbol{Y}(t))$
	is a Dual DO solution, then the equality \eqref{eq:DO-form-in-pf}
	implies
	\[
	(\dot{\boldsymbol{U}}(t),\dot{\boldsymbol{Y}}(t))=\Big(\frac{\mathrm{d}}{\mathrm{d}t}\big(\Theta(t)^{\top}\boldsymbol{V}(t)\big),\dot{\Theta}(t)^{\top}\boldsymbol{W}(t)+\Theta(t)^{\top}\dot{\boldsymbol{W}}(t)\Big),
	\]
	and from \eqref{eq:implied-gauge} we must have
	\begin{align*}
	0=\mathbb{E}[\boldsymbol{Y}(t)\dot{\boldsymbol{Y}}(t)^{\top}] & =\mathbb{E}\big[\Theta(t)^{\top}\boldsymbol{W}(t)\big(\dot{\Theta}(t)^{\top}\boldsymbol{W}(t)+\Theta(t)^{\top}\dot{\boldsymbol{W}}(t)\big)^{\top}\big]\\
	& =\Theta(t)^{\top}\mathbb{E}[\boldsymbol{W}(t)\boldsymbol{W}(t)^{\top}]\dot{\Theta}(t)+\Theta(t)^{\top}\mathbb{E}[\boldsymbol{W}(t)\dot{\boldsymbol{W}}(t)^{\top}]\Theta(t)\\
	& =\Theta(t)^{\top}\big(\dot{\Theta}(t)+\mathbb{E}[\boldsymbol{W}(t)\dot{\boldsymbol{W}}(t)^{\top}]\Theta(t)\big),
	\end{align*}
	where in the last line we used $\mathbb{E}[\boldsymbol{W}(t)\boldsymbol{W}(t)^{\top}]=I$.
	Using the orthonormality of $\Theta$ yields the equation
	\begin{equation}
	\dot{\Theta}(t)=-\mathbb{E}[\boldsymbol{W}(t)\dot{\boldsymbol{W}}(t)^{\top}]\Theta(t),\quad t\in(0,T)\ \text{ with }\Theta(0)=I.\label{eq:theta-eq}
	\end{equation}
	Now, from the assumptions we have
	\begin{equation}
	\int_{0}^{T}\!\|\mathbb{E}[\boldsymbol{W}(t)\dot{\boldsymbol{W}}(t)^{\top}]\|_{\mathrm{F}}\mathrm{d}t\leq
	\!\sup_{s\in[0,T]}\|\boldsymbol{W}(s)\|_{[L^{2}(\Omega)]^{S}}\int_{0}^{T}\|\dot{\boldsymbol{W}}(t)\|_{[L^{2}(\Omega)]^{S}}\mathrm{d}t<\infty,
	\label{eq:WW'integrable}
	\end{equation}
	where $\|\cdot\|_{\mathrm{F}}$ denotes the Frobenius norm, and thus
	$-\mathbb{E}[\boldsymbol{W}(\cdot)\dot{\boldsymbol{W}}(\cdot)^{\top}]\in\mathbb{R}^{S\times S}$  
	is integrable on $(0,T)$. Thus, from a standard fixed-point argument we obtain that a solution $\Theta\in C([0,T];\mathbb{R}^{S\times S})$
	of the integral equation $\Theta(t)=I-\int_{0}^{t}\mathbb{E}[\boldsymbol{W}(s)\dot{\boldsymbol{W}}(s)^{\top}]\Theta(s)\mathrm{d}s$,
	$t\in[0,T]$ uniquely exists in $C([0,T];\mathbb{R}^{S\times S})$. The solution $\Theta$
	thus obtained is absolutely continuous {on $[0,T]$},
	and satisfies \eqref{eq:theta-eq} a.e.\ on $(0,T)$ \cite[Theorem 1.17]{Miyadera.I_1992_book}.
	%
	\begin{comment} 
	\footnote{
	FN: Can't you establish directly global existence and uniqueness on $[0,T]$?
	\eqref{eq:theta-eq} is linear equation (in finite dimension).
	
	YK: As you say we can show a power is contraction etc.\ etc.
	The thing I thought was that the vector field $\mathbb{E}[\boldsymbol{W}(t)\dot{\boldsymbol{W}}(t)^{\top}]$ might not be continuous in $t$, and the standard ODE results(?) do not apply. Anyhow I agree we can show this directly.}
	YK: Now I do not really see how I can show this directly.
	\end{comment}
	The solution $\Theta$ thus obtained satisfies $\Theta(t)\in O(S)$ for all $t\in[0,T]$: we have a.e.\ on $[0,T]$
	\begin{align*}
	\frac{\mathrm{d}}{\mathrm{d}t}(\Theta^{\top}\Theta) & =-\Theta(t)^{\top}(\mathbb{E}[\boldsymbol{W}(t)\dot{\boldsymbol{W}}(t)^{\top}])^{\top}\Theta(t)-\Theta(t)^{\top}\mathbb{E}[\boldsymbol{W}(t)\dot{\boldsymbol{W}}(t)^{\top}]\Theta(t)\\
	& =\Theta(t)^{\top}\mathbb{E}[\boldsymbol{W}(t)\dot{\boldsymbol{W}}(t)^{\top}]\Theta(t)-\Theta(t)^{\top}\mathbb{E}[\boldsymbol{W}(t)\dot{\boldsymbol{W}}(t)^{{T}}]\Theta(t)=0,
	\end{align*}
	where in the penultimate equality we used $\mathbb{E}[\dot{\boldsymbol{W}}(t)\boldsymbol{W}(t)^{\top}]+\mathbb{E}[\boldsymbol{W}(t)\dot{\boldsymbol{W}}(t)^{\top}]=0$. 
	Thus, the absolute continuity of $t\mapsto \Theta(t)^{\top}\Theta(t)$ implies that  $\Theta^{\top}\Theta$ is constant on $[0,T]$, but from the initial condition we have $\Theta(t)^{\top}\Theta(t)=I$ for all $t\in[0,T]$. 
	With this solution $\Theta(t)\in O(S)$ of \eqref{eq:theta-eq}, let
	\begin{equation}
	\boldsymbol{U}(t):=\Theta(t)^{\top}\boldsymbol{V}(t),\ \text{and }\ \boldsymbol{Y}(t):=\Theta(t)^{\top}\boldsymbol{W}(t).
	\label{eq:def-UY}
	\end{equation}
	We claim that $(\boldsymbol{U}(t),\boldsymbol{Y}(t))$ is a Dual DO
	solution. First, we note that $\boldsymbol{U}$ is linearly independent,
	and that $\boldsymbol{Y}$ is orthonormal and satisfies the gauge
	condition. Indeed, we have $\mathrm{det}(\langle\boldsymbol{U}(t),\boldsymbol{U}(t)^{\top}\rangle)\neq0$,
	$\mathbb{E}[\boldsymbol{Y}(t)\boldsymbol{Y}(t)^{\top}]=I$, and further,
	\begin{align*}
	\mathbb{E}[\boldsymbol{Y}(t)\dot{\boldsymbol{Y}}(t)^{\top}] & =\mathbb{E}[\Theta(t)^{\top}\boldsymbol{W}(t)(\dot{\Theta}^{\top}(t)\boldsymbol{W}(t)+\Theta(t)^{\top}\dot{\boldsymbol{W}}(t))^{\top}]\\
	& =\Theta(t)^{\top}\mathbb{E}[\boldsymbol{W}(t)\boldsymbol{W}(t)^{\top}]\dot{\Theta}(t)+\Theta(t)^{\top}\mathbb{E}[\boldsymbol{W}(t)\dot{\boldsymbol{W}}(t)^{\top}]\Theta(t)\\
	& =\Theta(t)^{\top}\big(\dot{\Theta}(t)+\mathbb{E}[\boldsymbol{W}(t)\dot{\boldsymbol{W}}(t)^{\top}]\Theta(t)\big)=0,
	\end{align*}
	where in the penultimate line we used $\mathbb{E}[\boldsymbol{W}(t)\boldsymbol{W}(t)^{\top}]=I$.
	Then, noting that $\boldsymbol{U}(t)^{\top}\boldsymbol{Y}(t)=\boldsymbol{V}(t)^{\top}\boldsymbol{W}(t)=u_{S}(t)$
	satisfies the original equation \eqref{eq:original-eq-with-proj},
	{from the derivation of the Dual DO equation \eqref{eq:working-eq}
		(see \cite{Musharbash.E_Nobile_2018_Dual}, also \cite{Musharbash.E_etal_2015_SISC,Sapsis.T_Lermusiaux_2009_DO})}
	we conclude that $(\boldsymbol{U}(t),\boldsymbol{Y}(t))$ satisfies
	\eqref{eq:working-eq}.
	From \eqref{eq:def-UY}, we see that on the compact interval $[0,T]$ the functions $t\mapsto\boldsymbol{U}(t)\in [\mathcal{H}]^S$ and $t\mapsto\boldsymbol{Y}(t)\in [L^2(\Omega)]^S$ are absolutely continuous, and thus $(\boldsymbol{U}(t),\boldsymbol{Y}(t))$
	is a Dual DO strong solution.
	
	To see the uniqueness of the Dual DO solution, we note
	that if $(\hat{\boldsymbol{U}}(t),\hat{\boldsymbol{Y}}(t))$ is another 
	Dual DO solution, then from Corollary~\ref{cor:dual-DO-form}
	we must have \[(\hat{\boldsymbol{U}}(t),\hat{\boldsymbol{Y}}(t))=(\hat{\Theta}(t)^{\top}\boldsymbol{V}(t),\hat{\Theta}(t)^{\top}\boldsymbol{W}(t))\quad 
	\text{ for a unique }\ \hat{\Theta}(t)\in O(S).
	\]
	But following the same argument
	as above, $\hat{\Theta}(t)$ must be a solution of \eqref{eq:theta-eq},
	which is unique. Thus, $(\hat{\boldsymbol{U}},\hat{\boldsymbol{Y}})=(\Theta^{\top}\boldsymbol{V},\Theta^{\top}\boldsymbol{W})=(\boldsymbol{U},\boldsymbol{Y})$.
	\end{proof}
We are ready to state the following equivalence of the original problem
\eqref{eq:original-eq-with-proj} and the Dual DO formulation (Definitions~\ref{def:Dual-DO-strong}--\ref{def:Dual-DO-classical}).
\begin{proposition}
	\label{prop:existence-equiv}
	\begin{sloppypar}
		Suppose that the solution $u_{S}$
		of the original equation \eqref{eq:original-eq-with-proj} uniquely
		exists in the strong sense (resp.~the classical sense). 
		Then, given the decomposition $(\boldsymbol{U}_{0},\boldsymbol{Y}_{0})\in[\mathcal{H}]^{S}\times[L^{2}(\Omega)]^{S}$
		of the initial condition ${u_{0S}}=\boldsymbol{U}_{0}^{\top}\boldsymbol{Y}_{0}\in M_{S}$,
		the Dual DO solution with the initial condition $(\boldsymbol{U}_{0},\boldsymbol{Y}_{0})$
		uniquely exists in the strong sense (resp.~the classical sense).
		Conversely, the unique existence of the Dual DO approximation in the
		strong sense (resp.~the classical sense) implies the unique existence
		of the solution of the original equation \eqref{eq:original-eq-with-proj}.
	\end{sloppypar}
\end{proposition}

\begin{proof}
	The first direction is a direct consequence of the previous
	lemma for strong solutions.
	
	Suppose that the Dual DO approximation $(\boldsymbol{U}(t),\boldsymbol{Y}(t))_{t\in[0,T]}$
	uniquely exists in the strong sense. 
	Then, from the {derivation}
	of the Dual DO equation \eqref{eq:working-eq}, $t\mapsto\boldsymbol{U}^{\top}(t)\boldsymbol{Y}(t)\in M_{S}$
	is a solution of the original equation \eqref{eq:original-eq-with-proj}.
	
	Now, we show the uniqueness. Suppose that $t\mapsto\hat{u}_{S}(t)\neq\boldsymbol{U}^{\top}(t)\boldsymbol{Y}(t)$
	satisfies the original equation \eqref{eq:original-eq-with-proj}.
	From Lemma \ref{lem:unique-lift}, there exists a unique Dual DO approximation
	$(\hat{\boldsymbol{U}},\hat{\boldsymbol{Y}})$ associated with
	$\hat{u}_{S}$ {and the decomposition $\hat{u}_{S}(0)=\boldsymbol{U}_0^{\top}\boldsymbol{Y}_0$}, i.e.\ $(\hat{\boldsymbol{U}}(t),\hat{\boldsymbol{Y}}(t))$
	is a solution of the Dual DO equation \eqref{eq:working-eq}. 
	But from the assumption we must have $(\hat{\boldsymbol{U}}(t),\hat{\boldsymbol{Y}}(t))=(\boldsymbol{U}(t),\boldsymbol{Y}(t))$,
	$t\in[0,T]$ and therefore $\hat{\boldsymbol{U}}(t)^{\top}\hat{\boldsymbol{Y}}(t)=\hat{u}_{S}(t)=\boldsymbol{U}(t)^{\top}\boldsymbol{Y}(t)=u_{S}(t)$,
	a contradiction.
	
	The argument for the classical solution is analogous.
	\end{proof}

\subsection{Assumptions}\label{sec:Assumptions}

In view of Proposition \ref{prop:existence-equiv}, we establish the
unique existence of the Dual DO approximation. We work under the following
assumptions. Assumptions~\ref{assu:Lam} and \ref{assu:lip-F-H}
will be used for the existence in the strong sense, and in addition,
Assumption~\ref{assu:lip-F-D(Lam)} will be used for the classical
sense. Further, the stability Assumptions~\ref{assu:stab} and \ref{assu:glob-lin-F-D(A)}
will be used to establish the extendability of the strong solution,
and respectively the classical solution, to the maximal time interval.
\begin{assumption}
	\label{assu:Lam}{$\Lambda:D_{\mathcal{H}}(\Lambda)\subset\mathcal{H}\rightarrow\mathcal{H}$
		is a closed linear operator that is densely defined in $\mathcal{H}$.
		Furthermore, $\Lambda$ is the infinitesimal generator of the $C_{0}$
		semigroup $\mathrm{e}^{t\Lambda}$ satisfying $\|\mathrm{e}^{t\Lambda}\|_{\mathcal{H}\to\mathcal{H}}\leq K_{\Lambda}\mathrm{e}^{-\lambda t}$
		for $t\geq0$, with constants $K_{\Lambda}\ge1$ and $\lambda\ge0$.}
\end{assumption}

\begin{assumption}
	\label{assu:lip-F-H}The mapping $F:L^{2}(\Omega;\mathcal{H})\rightarrow L^{2}(\Omega;\mathcal{H})$
	is locally Lipschitz continuous on \textup{$L^{2}(\Omega;\mathcal{H})$}
	in the following sense: for every $r>0$ and every $v_{0}\in L^{2}(\Omega;\mathcal{H})$
	such that $\|v_{0}\|_{L^{2}(\Omega;\mathcal{H})}\leq q$, there exists
	a constant $C_{q,r}>0$ such that
	\[
	\|F(w)-F(w')\|_{L^{2}(\Omega;\mathcal{H})}\leq C_{q,r}\|w-w'\|_{L^{2}(\Omega;\mathcal{H})}
	\]
	holds for all $w,w'\in L^{2}(\Omega;\mathcal{H})$ with $\|w-v_{0}\|_{L^{2}(\Omega;\mathcal{H})}\le r$,
	$\|w'-v_{0}\|_{L^{2}(\Omega;\mathcal{H})}\le r$. Furthermore, we
	assume $\|F(v_{0})\|_{L^{2}(\Omega;\mathcal{H})}<C'{}_{q}<\infty$.
\end{assumption}

In the above assumption, note that given the first condition, the
second condition is implied by $\|F(a)\|_{L^{2}(\Omega;\mathcal{H})}<\infty$
for a point $a\in L^{2}(\Omega;\mathcal{H})$.

In practice, one might be interested in the classical solution. To
establish the existence of the Dual DO approximation in the classical
sense, we use the following further regularity of $F$.
\begin{assumption}
	\label{assu:lip-F-D(Lam)}In addition to Assumption \ref{assu:lip-F-H},
	assume that for every $r>0$ and every $v_{0}\in L^{2}(\Omega;\mathcal{H})$
	with $\Lambda v_{0}\in L^{2}(\Omega;\mathcal{H})$ such that $\|\Lambda v_{0}\|_{L^{2}(\Omega;\mathcal{H})}\leq q$,
	there exists a constant $C_{q,r}>0$ such that
	\[
	\|\Lambda(F(w)-F(w'))\|_{L^{2}(\Omega;\mathcal{H})}\leq C{}_{q,r}\|\Lambda(w-w')\|_{L^{2}(\Omega;\mathcal{H})}
	\]
	holds for any $w,w'\in L^{2}(\Omega;\mathcal{H})$ satisfying $\Lambda w,\Lambda w'\in L^{2}(\Omega;\mathcal{H})$
	with $\|\Lambda(w-v_{0})\|_{L^{2}(\Omega;\mathcal{H})}\le r,\|\Lambda(w'-v_{0})\|_{L^{2}(\Omega;\mathcal{H})}\le r$\textup{.}
	Further, assume $\|\Lambda F(v_{0})\|_{L^{2}(\Omega;\mathcal{H})}<C'{}_{q}<\infty$. 
\end{assumption}

Since $\Lambda$ is closed, $D_{\mathcal{H}}(\Lambda)$ admits a Hilbert
space structure with respect to the graph inner product ${\langle\cdot,\cdot\rangle+\langle\Lambda\cdot,\Lambda\cdot\rangle}$,
which we denote $\mathcal{V}$. Then, Assumptions~\ref{assu:lip-F-H}--\ref{assu:lip-F-D(Lam)}
{imply that for a constant $\tilde{C}_{q,r}>0$ we have
	\[
	\|F(w)-F(w')\|_{L^{2}(\Omega;\mathcal{V})}\leq\tilde{C}_{q,r}\|w-w'\|_{L^{2}(\Omega;\mathcal{V})}
	\]
	for any $w,w'\in\mathcal{V}$ satisfying $\|w-v_{0}\|_{L^{2}(\Omega;\mathcal{V})}\le r,\|w'-v_{0}\|_{L^{2}(\Omega;\mathcal{V})}\le r$,
	and moreover, $\|F(v_{0})\|_{L^{2}(\Omega;\mathcal{V})}<\tilde{C}'{}_{q}<\infty$.}

The following uniform stability condition will be used to establish
the existence of a Dual DO solution in the maximal interval, in the
strong sense. Here, uniform means that the constant $C_{\Lambda,F}$
below is independent of bounds of $v$.

\begin{assumption}
	\label{assu:stab}The pair $(\Lambda,F)$ satisfies the following:
	for every $v\in L^{2}(\Omega;\mathcal{H})$ such that $\Lambda v\in L^{2}(\Omega;\mathcal{H})$
	we have
	\[
	\mathbb{E}[\langle\Lambda(v)+F(v),v\rangle]\leq C_{\Lambda,F}(1+\|v\|_{L^{2}(\Omega;\mathcal{H})}^{2}).
	\]
\end{assumption}

For example, this condition holds when $\Lambda$ satisfies
$\langle\Lambda x,x\rangle\leq0$ for $x\in D_{\mathcal{H}}(\Lambda)$
and $F$ satisfies the uniform linear growth condition $\|F(v)\|_{L^{2}(\Omega;\mathcal{H})}\leq C'_{F}(1+\|v\|_{{L^{2}(\Omega;\mathcal{H})}})$
for some $C'_{F}>0$.

To establish the existence of the DO solution in the maximal
interval in the classical sense, we use the following {stronger uniform stability}
condition, where we again note that the constant is independent of
bounds of $v$. 
\begin{assumption}
	\label{assu:glob-lin-F-D(A)} For every $v\in L^{2}(\Omega;\mathcal{H})$
	such that $\Lambda v\in L^{2}(\Omega;\mathcal{H})$ we have
	\[
	\|\Lambda F(v)\|_{L^{2}(\Omega;\mathcal{H})}\le C_{F}(1+\|\Lambda v\|_{L^{2}(\Omega;\mathcal{H})}),\quad\text{where }C_{F}>0\ \text{is independent of }v.
	\]
\end{assumption}

The following examples satisfy the above assumptions.
\begin{example}
	For a bounded domain $D\subset\mathbb{R}^{d}$, let $\mathcal{H}=L^{2}(D)$.
	Further, let $\tilde{\Lambda}$ be a second order uniformly elliptic
	differential operator with zero Dirichlet boundary condition. For
	the non-linear term, let $a,b\in L^{\infty}(\Omega;L^{\infty}(D))$,
	$c\in L^{2}(\Omega;L^{2}(D))$, and let $f\colon\mathbb{R}\to\mathbb{R}$
	be a differentiable function such that $\sup_{s\in\mathbb{R}}|f'(s)|<\infty$.
	Consider the following multiplicative and additive noise:
	\[
	\tilde{F}(v):=a\cdot f(v\cdot b)+c,\ \text{for }v\in L^{2}(\Omega;L^{2}(D)),
	\]
	where $\cdot$ denotes the point-wise multiplication. Then, the pair
	$(\tilde{\Lambda},\tilde{F})$ satisfies Assumptions~\ref{assu:Lam},
	\ref{assu:lip-F-H}, and~\ref{assu:stab}.
\end{example}

\begin{example}
	Let $f(x)=x$. With $a\in\!\!\; L^{\infty}(\Omega;W^{\infty,2}(D))$
	and $c\in\!\!\;  L^{2}(\Omega;L^{2}(D))$, let
	\[
	\tilde{\tilde{F}}(v):=a\cdot v+c,\ \text{for }v\in L^{2}(\Omega;L^{2}(D)).
	\]
	Then, the pair $(\tilde{\Lambda},\tilde{\tilde{F}})$ satisfies Assumptions
	\ref{assu:Lam}--\ref{assu:glob-lin-F-D(A)}.
\end{example}

\subsection{{On the choice of the Dual DO formulation}\label{subsec:Discussion-on-gauge}}

{To establish uniqueness and existence of the DLR approximation
	we work with the Dual DO formulation \eqref{eq:working-eq}. We have
	chosen this formulation with care. This section provides a discussion
	on choosing a good formulation.}

The DLR approach to the stochastic dynamical system such as \eqref{eq:exact-eq}
was first introduced by Sapsis and Lermusiaux \cite{Sapsis.T_Lermusiaux_2009_DO}.
The formulation they introduced is called the Dynamically Orthogonal
(DO) formulation: they imposed the orthogonality of the spatial basis.
Musharbash et al.\ \cite{Musharbash.E_etal_2015_SISC} pointed out
that the DO approximation can be related to the MCTDH method, by considering
the so-called dynamically double orthogonal (DDO) formulation: yet
another equivalent formulation of the DLR approach. Through this relation
of the DDO approximation to the MCTDH method, Musharbash et al.\ further
developed an error estimate of the DO method. The error analysis obtained
by Musharbash et al.\ was partially built upon results regarding
the MCTDH method.

A reasonable strategy to establish the existence of the DLR approximation
would thus be to establish the existence of the DDO approximation.
Namely, following the argument of Koch and Lubich \cite{Koch.O_Lubich_2007_regularity_existence},
it is tempting to apply the gauge condition defined by the differential
operator $\Lambda$ to the DDO formulation. It turns out that this approach
does not work, since the aforementioned
gauge condition turns out to be vacuous unless $\Lambda$ is skew-symmetric, as we illustrate hereafter.

In the DDO formulation, we seek an approximant of the form
\[
u_{S}(t)=\tilde{\boldsymbol{U}}^{\top}(t)A(t)\boldsymbol{Y}(t),
\]
where $\tilde{\boldsymbol{U}}(t)=(U_{1}(t),\dots,U_{S}(t))^{\top}$,
and $\boldsymbol{Y}(t)=(Y_{1}(t),\dots,Y_{S}(t))^{\top}$ are orthonormal
in $\mathcal{H}$, and in $L^{2}(\Omega)$ respectively; and $A(t)\in\mathbb{R}^{S\times S}$
is a full-rank matrix. The triplet $(\tilde{\boldsymbol{U}}(t),A(t),\boldsymbol{Y}(t))$
is given as a solution of the set of equations:
\begin{align}
\frac{\mathrm{d}}{\mathrm{d}t}A & =\mathbb{E}\left[\left\langle \Lambda(u_{S})+F(u_{S}),\tilde{\boldsymbol{U}}\right\rangle \boldsymbol{Y}^{\top}\right],\nonumber \\
A^{\top}\frac{\mathrm{d}\tilde{\boldsymbol{U}}}{\mathrm{d}t} & =(I-P_{\tilde{\boldsymbol{U}}})A^{\top}\Lambda(\tilde{\boldsymbol{U}})+(I-P_{\tilde{\boldsymbol{U}}})\mathbb{E}\left[\boldsymbol{Y}\big(F(u_{S})\big)\right],\label{eq:DDO-U}\\
A\frac{\partial\boldsymbol{Y}}{\partial t} & =(I-P_{\boldsymbol{Y}})\left\langle \Lambda(\tilde{\boldsymbol{U}}^{\top})A\boldsymbol{Y}+F(u_{S}),\tilde{\boldsymbol{U}}\right\rangle ,\nonumber 
\end{align}
where $P_{\tilde{\boldsymbol{U}}}:\mathcal{H}\rightarrow\mathrm{span}\{\tilde{U}_{j}:j=1,\dotsc,S\}$
is the $\mathcal{H}$-orthogonal projection onto $\mathrm{span}\{\tilde{U}_{j}:j=1,\dotsc S\}$,
and $P_{\boldsymbol{Y}}:L^{2}(\Omega)\rightarrow\mathrm{span}\{Y_{j}:j=1,\dotsc,S\}$
is the $L^{2}(\Omega)$-orthogonal projection onto $\mathrm{span}\{Y_{j}:j=1,\dotsc S\}$.
These equations are derived using the orthonormality assumption on
$(\tilde{\boldsymbol{U}},\boldsymbol{Y})$ together with the {gauge}
conditions
\begin{equation}
\langle\frac{\partial}{\partial t}\tilde{\boldsymbol{U}},\tilde{\boldsymbol{U}}^{\top}\rangle=0\text{ and }\mathbb{E}\left[\Big(\frac{\partial}{\partial t}\boldsymbol{Y}\Big)\boldsymbol{Y}^{\top}\right]=0,\label{eq:gauge-on-both}
\end{equation}
see \cite[(3.14)--(3.17)]{Musharbash.E_etal_2015_SISC}.

We note that in the equation \eqref{eq:DDO-U} for $\tilde{\boldsymbol{U}}$
we have the composition of the unbounded operator $\Lambda$ and the
projection operator $P_{\tilde{\boldsymbol{U}}}$, where we note that
the map $\tilde{\boldsymbol{U}}\mapsto P_{\tilde{\boldsymbol{U}}}$
is non-linear. Koch and Lubich \cite{Koch.O_Lubich_2007_regularity_existence}
had a similar situation in the MCTDH setting. As outlined above, they
got away with this problem by considering a different gauge condition.
We will explain below an analogous strategy and why it does not work
in our setting.

First, from the orthonormality condition on $\tilde{\boldsymbol{U}}$
it is necessary to have $\frac{\mathrm{d}}{\mathrm{d}t}\langle\tilde{\boldsymbol{U}},\tilde{\boldsymbol{U}}^{\top}\rangle=0$.
The above gauge condition \eqref{eq:gauge-on-both} on $\tilde{\boldsymbol{U}}$
is sufficient for this to hold. But since
\[
\frac{\mathrm{d}}{\mathrm{d}t}\langle\tilde{\boldsymbol{U}},\tilde{\boldsymbol{U}}^{\top}\rangle=\langle\frac{\partial}{\partial t}\tilde{\boldsymbol{U}},\tilde{\boldsymbol{U}}^{\top}\rangle+\langle\tilde{\boldsymbol{U}},\frac{\partial}{\partial t}\tilde{\boldsymbol{U}}^{\top}\rangle,
\]
the solution $\tilde{\boldsymbol{U}}$ stays orthonormal if and only
if we impose the gauge condition $\langle\frac{\partial}{\partial t}\tilde{\boldsymbol{U}},\tilde{\boldsymbol{U}}^{\top}\rangle=-\langle\tilde{\boldsymbol{U}},\frac{\partial}{\partial t}\tilde{\boldsymbol{U}}^{\top}\rangle$.
Koch and Lubich \cite{Koch.O_Lubich_2007_regularity_existence} noted
this, and to establish an existence result they considered a suitable
gauge condition, which enabled them to take the differential operator
out of the projection. The gauge condition that is formally
analogous to \cite{Koch.O_Lubich_2007_regularity_existence} may be
given as
\[
\langle\frac{\partial}{\partial t}\tilde{\boldsymbol{U}},\tilde{\boldsymbol{U}}^{\top}\rangle=\langle\Lambda\tilde{\boldsymbol{U}},\tilde{\boldsymbol{U}}^{\top}\rangle,
\]
for $\Lambda$ not necessarily skew-symmetric. One can check that
this condition formally allows us to take the operator $\Lambda$
out of the projection $P_{\tilde{\boldsymbol{U}}}$, but for example
when $\Lambda$ is self-adjoint, the solution $\tilde{\boldsymbol{U}}$
will not stay orthonormal. This is {not acceptable}, since we
use the orthonormality to derive the equations \eqref{eq:DDO-U},
and thus we necessarily have to consider a different gauge condition
or a different formulation.
\section{Parameter equation\label{sec:param-eq}}
This section introduces the parameter equation, for which we establish
the unique existence of the solution later in Section~\ref{sec:Existence-and-Regularity}.  Consider the direct sum of the Hilbert spaces ${\mathcal{X}:=}[\mathcal{H}]^{S}\oplus[L^{2}(\Omega)]^{S}$ 
equipped with the inner product
\[
\langle(\hat{\boldsymbol{U}},\hat{\boldsymbol{Y}}),(\hat{\boldsymbol{V}},\hat{\boldsymbol{W}})\rangle_{\mathcal{X}}:=\langle\hat{\boldsymbol{U}},\hat{\boldsymbol{V}}\rangle_{[\mathcal{H}]^{S}}+\langle\hat{\boldsymbol{Y}},\hat{\boldsymbol{W}}\rangle_{[L^{2}(\Omega)]^{S}}.
\]
{In what follows, we redefine the operator $\Lambda$ as $\Lambda\colon D_{\mathcal{H}}(\Lambda)\subset[\mathcal{H}]^{S}\to[\mathcal{H}]^{S}$,
	$\boldsymbol{U}\mapsto(\Lambda U_{1},\dots,\Lambda U_{S})=:\Lambda\boldsymbol{U}$
	for $\boldsymbol{U}\in D_{\mathcal{H}}(\Lambda)\subset[\mathcal{H}]^{S}$.}
We define the linear operator $A:\mathcal{X}\rightarrow\mathcal{X}$
by 
\[
A(\hat{\boldsymbol{U}},\hat{\boldsymbol{Y}})=(\Lambda\hat{\boldsymbol{U}},0)\quad\text{for }(\hat{\boldsymbol{U}},\hat{\boldsymbol{Y}})\in\mathcal{X},
\]
with $D(A)=D_{\mathcal{H}}(\Lambda)\oplus[L^{2}(\Omega)]^{S}$. Further,
we define the mapping $G:D(G)\subset\mathcal{X}\rightarrow\mathcal{X}$
by
\begin{align}
G(\hat{\boldsymbol{U}},\hat{\boldsymbol{Y}}) & :=\left([G_{1}(\hat{\boldsymbol{Y}})](\hat{\boldsymbol{U}}),[G_{2}(\hat{\boldsymbol{U}})](\hat{\boldsymbol{Y}})\right)\nonumber \\
& :=\left(\mathbb{E}\left[F(\hat{\boldsymbol{U}}^{\top}\hat{\boldsymbol{Y}})\hat{\boldsymbol{Y}}\right],(I-P_{\hat{\boldsymbol{Y}}})\big(\langle F(\hat{\boldsymbol{U}}^{\top}\hat{\boldsymbol{Y}}),Z_{\hat{\boldsymbol{U}}}^{-1}\hat{\boldsymbol{U}}\rangle\big)\right),\label{eq:def-G}
\end{align}
where $D(G):=\{(\hat{\boldsymbol{U}},\hat{\boldsymbol{Y}})\in\mathcal{X}\mid Z_{\hat{\boldsymbol{U}}}^{-1}\text{ exists}\}$.
Then, the Dual DO solution, if it exists, satisfies the following
Cauchy problem for a semi-linear abstract evolution equation in~$\mathcal{X}$:
\begin{equation}
\left\{ \begin{array}{rl}
\frac{\mathrm{d}}{\mathrm{d}t}(\boldsymbol{U},\boldsymbol{Y}) & =A(\boldsymbol{U},\boldsymbol{Y})+G(\boldsymbol{U},\boldsymbol{Y})\quad\text{for}\ t>0,\\
(\boldsymbol{U}(0),\boldsymbol{Y}(0)) & =(\boldsymbol{U}_{0},\boldsymbol{Y}_{0}),
\end{array}\right.\label{eq:Cauchy}
\end{equation}
where the initial condition $(\boldsymbol{U}_{0},\boldsymbol{Y}_{0})\in\mathcal{X}$
satisfies suitable assumptions detailed below. Conversely, later in
Section \ref{sec:Existence-and-Regularity} we will see that the strong
solution of this Cauchy problem is a Dual DO {solution of \eqref{eq:original-eq-with-proj}, so, in particular, $\mathbb{E}[\boldsymbol{Y}\boldsymbol{Y}^\top]=I$ and $\mathbb{E}[\dot{\boldsymbol{Y}}\boldsymbol{Y}^\top]=0$}, and thus
in view of Proposition~\ref{prop:existence-equiv}, it is a solution
of the original problem \eqref{eq:original-eq-with-proj}.

We first establish the unique existence of the mild solution of the
problem \eqref{eq:Cauchy}:
\begin{align*}
\boldsymbol{U}(t) & =e^{t\Lambda}\boldsymbol{U}(0)+\int_{0}^{t}e^{(t-\tau)\Lambda}\big[G_{1}\big(\boldsymbol{Y}(\tau)\big)\big]\big(\boldsymbol{U}(\tau)\big)\mathrm{d}\tau,\\
\boldsymbol{Y}(t) & =\boldsymbol{Y}(0)+\int_{0}^{t}\big[G_{2}\big(\boldsymbol{U}(\tau)\big)\big]\big(\boldsymbol{Y}(\tau)\big)\mathrm{d}\tau.
\end{align*}
%which may be thought of as a generalised Dual DO approximation: with a suitably regular initial condition, the mild solution is a strong solution of \eqref{eq:Cauchy}, which, {as we will see later in Section \ref{sec:Existence-and-Regularity},} gives the Dual DO approximation in the strong sense. 
We will use the following result, which is a variation of
a standard local existence and uniqueness theorem for mild solutions, e.g.\ see \cite[Theorem 6.1.4]{Pazy.A_1983_book} or \cite[Theorem 46.1]{Sell.G_You_2013_book}, adapted to our setting.
\begin{proposition}
	\label{prop:gen-mild-local}
	Suppose that 
	the operator $A\colon D(A)\subset\mathcal{X}\rightarrow\mathcal{X}$
	generates a $C_{0}$ semigroup $\mathrm{e}^{tA}$, $t\ge0$
	on $\mathcal{X}$.
	Suppose further that the mapping $G\colon\mathcal{X}\rightarrow\mathcal{X}$
	is locally Lipschitz continuous on $\mathcal{X}$ in the
	following sense: 
	for an element $(\hat{\boldsymbol{U}},\hat{\boldsymbol{Y}})\in\mathcal{X}$
	with $\alpha\ge\|\hat{\boldsymbol{U}}\|_{[\mathcal{H}]^{S}}$
	and $\beta\geq\|\hat{\boldsymbol{Y}}\|_{[L^{2}(\Omega)]^{S}}$, there
	exists $r=r(\hat{\boldsymbol{U}},\hat{\boldsymbol{Y}})>0$ and $C_{\alpha,\beta}>0$
	such that
	\[
	\|G(\boldsymbol{V},\boldsymbol{W})-G(\boldsymbol{V}',\boldsymbol{W}')\|_{\mathcal{X}}\leq C_{\alpha,\beta}\|(\boldsymbol{V},\boldsymbol{W})-(\boldsymbol{V}',\boldsymbol{W}')\|_{\mathcal{X}}
	\]
	holds for all $(\boldsymbol{V},\boldsymbol{W}),(\boldsymbol{V}',\boldsymbol{W}')\in\mathcal{X}$
	with $\|\boldsymbol{W}-\hat{\boldsymbol{Y}}\|_{[L^{2}(\Omega)]^{S}}\leq r$
	and $\|\boldsymbol{W}'-\hat{\boldsymbol{Y}}\|_{[L^{2}(\Omega)]^{S}}\leq r$;
	$\|\boldsymbol{V}-\hat{\boldsymbol{U}}\|_{[\mathcal{H}]^{S}}\leq r$
	and $\|\boldsymbol{V}'-\hat{\boldsymbol{U}}\|_{[\mathcal{H}]^{S}}\leq r$.
	Further, suppose that for some $C'_{\alpha,\beta}>0$ we have
	\[
	\|G(\hat{\boldsymbol{U}},\hat{\boldsymbol{Y}})\|_{\mathcal{X}}\leq C'_{\alpha,\beta}.
	\]
	Then, the problem \eqref{eq:Cauchy} starting at $t_0\geq0$ with the initial condition $(\hat{\boldsymbol{U}},\hat{\boldsymbol{Y}})\in{\mathcal{X}}$: 
	\begin{equation*}
	\left\{ \begin{array}{rl}
	\frac{\mathrm{d}}{\mathrm{d}t}(\boldsymbol{U},\boldsymbol{Y}) & =A(\boldsymbol{U},\boldsymbol{Y})+G(\boldsymbol{U},\boldsymbol{Y})\quad\text{for}\ t>t_0,\\
	(\boldsymbol{U}(t_0),\boldsymbol{Y}(t_0)) & =(\hat{\boldsymbol{U}},\hat{\boldsymbol{Y}}),
	\end{array}\right.
	\label{eq:Cauchy-t0}
	\end{equation*}
	has a unique mild solution on an interval of length $\delta\in(0,1]$,
	where $\delta$ depends on $\alpha$, $\beta$, $\sup_{s{\in[t_0,t_0+1]}}\|\mathrm{e}^{sA}\|$,
	and $r=r(\hat{\boldsymbol{U}},\hat{\boldsymbol{Y}})$.
	\begin{comment}\footnote{FN: (Regarding $\delta\in(0,1]$ ) Why $1$?
	YK: It is just for a technical reason/convenience. The cut-off upper
	bound can be anything e.g.\ $2019$, cf., footnote~\ref{fn:on-threshold-1}.}
	\end{comment}
\end{proposition}
To invoke this proposition, we start with checking that the operator $A$ defined above generates
a $C_{0}$ semigroup.
\begin{proposition}
	Let Assumption \ref{assu:Lam} hold. Then, the operator $A\colon D(A)\subset\mathcal{X}\rightarrow\mathcal{X}$
	generates a $C_{0}$ semigroup $\mathrm{e}^{tA}$, $t\ge0$
	on $\mathcal{X}$ with the bound $\|\mathrm{e}^{tA}\|_{\mathcal{X}\to\mathcal{X}}\leq K_{\Lambda}$.
\end{proposition}

\begin{proof}
	We note that $D(A)=D_{\mathcal{H}}(\Lambda)\oplus[L^{2}(\Omega)]^{S}$
	is dense in $\mathcal{X}$. Further, the closedness of $\Lambda\colon D_{\mathcal{H}}(\Lambda)\subset[\mathcal{H}]^{S}\to[\mathcal{H}]^{S}$
	implies that $A:D(A)\subset\mathcal{X}\rightarrow\mathcal{X}$ is
	closed. 
	
	We will invoke the Hille--Yosida theorem, see for example \cite[Theorem {1.5.2}]{Pazy.A_1983_book}.
	From Assumption \ref{assu:Lam}, every $\mu>0$ is in the resolvent
	set of $\Lambda$. Thus, $(\mu I-\Lambda)^{-1}\colon[\mathcal{H}]^{S}\to[\mathcal{H}]^{S}$
	as well as $(\mu I-0)^{-1}=\frac{1}{\mu}\colon[L^{2}(\Omega)]^{S}\to[L^{2}(\Omega)]^{S}$
	are well-defined, and so is $(\mu I-A)^{-1}$. For any $(\hat{\boldsymbol{U}},\hat{\boldsymbol{Y}})\in{\mathcal{X}}$,
	${n\in\mathbb{N}}$ we have
	\[
	\|(\mu I-A)^{-{n}}(\hat{\boldsymbol{U}},\hat{\boldsymbol{Y}})\|_{\mathcal{X}}^{2}=\|(\mu I-\Lambda)^{-{n}}\hat{\boldsymbol{U}}\|_{[\mathcal{H}]^{S}}^{2}+\frac{1}{\mu^{2{n}}}\|\hat{\boldsymbol{Y}}\|_{[L^{2}(\Omega)]^{S}}^{2},
	\]
	but Assumption \ref{assu:Lam} implies ${\|(\mu I-\Lambda)^{-n}\|_{[\mathcal{H}]^{S}}\leq K_{\Lambda}/\mu^{n}}$,
	and thus we obtain
	\begin{align*}
	\|(\mu I-A)^{-{n}}(\hat{\boldsymbol{U}},\hat{\boldsymbol{Y}})\|_{\mathcal{X}}^{2} & \leq\frac{{K_{\Lambda}^{2}}}{\mu^{2{n}}}\|\hat{\boldsymbol{U}}\|_{[\mathcal{H}]^{S}}^{2}+\frac{1}{\mu^{2{n}}}\|\hat{\boldsymbol{Y}}\|_{[L^{2}(\Omega)]^{S}}^{2}\\
	& {\leq}\frac{{K_{\Lambda}^{2}}}{\mu^{2{n}}}\|(\hat{\boldsymbol{U}},\hat{\boldsymbol{Y}})\|_{\mathcal{X}}^{2},
	\end{align*}
	and thus $\|(\mu I-A)^{-n}\|_{\mathcal{X}\to\mathcal{X}}\leq K_{\Lambda}/\mu^{n}$.
	In view of the Hille--Yosida theorem the statement follows.
	\end{proof}
{Furthermore}, we establish a Lipschitz continuity of the
non-linear term $G$. We start with the Lipschitz continuity of the projection
operator.
\begin{lemma}
	\label{lem:Wedin-less-1}For $\hat{\boldsymbol{Y}}=(\hat{Y}_{1},\dots,\hat{Y}_{S})^{\top}\in[L^{2}(\Omega)]^{S}$,
	suppose that the smallest eigenvalue $\sigma_{\hat{\boldsymbol{Y}}}$
	of the Gram matrix $E[\hat{\boldsymbol{Y}}\hat{\boldsymbol{Y}}^{\top}]$
	is non-zero. Further, let $\kappa\in(0,\overline{\kappa})$ be given,
	where with $\beta\geq\|\hat{\boldsymbol{Y}}\|_{[L^{2}(\Omega)]^{S}}$,
	we let
	\begin{equation}
	\overline{\kappa}
	:=\overline{\kappa}(\sigma_{\hat{\boldsymbol{Y}}},\beta)
	:=\frac{1}{2}\big(-\beta+\sqrt{\beta^{2}+{\sigma_{\hat{\boldsymbol{Y}}}}}\big).
	\label{eq:ub-kappa}
	\end{equation}
	Then, we have
	\begin{equation}
	\|(I-P_{\hat{\boldsymbol{W}}'})P_{\hat{\boldsymbol{W}}}\|_{[L^{2}(\Omega)]^{S}\to[L^{2}(\Omega)]^{S}}\leq C_{\kappa,\beta,\sigma_{\hat{\boldsymbol{Y}}}}\|(\hat{\boldsymbol{W}}-\hat{\boldsymbol{W}}')\|_{[L^{2}(\Omega)]^{S}}<1\label{eq:Wedin-less-1}
	\end{equation}
	for any $\hat{\boldsymbol{W}},\hat{\boldsymbol{W}}'\in[L^{2}(\Omega)]^{S}$
	with $\|\hat{\boldsymbol{W}}-\hat{\boldsymbol{Y}}\|_{[L^{2}(\Omega)]^{S}}\leq\kappa$,
	$\|\hat{\boldsymbol{W}}'-\hat{\boldsymbol{Y}}\|_{[L^{2}(\Omega)]^{S}}\leq\kappa$,
	where $C_{\kappa,\beta,\sigma_{\hat{\boldsymbol{Y}}}}
	:={2({\kappa+\beta})/\sigma_{\hat{\boldsymbol{Y}}}}$. 
\end{lemma}
\begin{proof}
	We first show that the smallest eigenvalue $\sigma_{\hat{\boldsymbol{W}}}$
	of the Gram matrix $\mathbb{E}[\hat{\boldsymbol{W}}\hat{\boldsymbol{W}}^{\top}]$
	is positive, and thus in particular $\mathbb{E}[\hat{\boldsymbol{W}}\hat{\boldsymbol{W}}^{\top}]$
	is non-singular. {We have
		\[
		\frac{-\beta+\sqrt{\beta^{2}+\frac{\sigma_{\hat{\boldsymbol{Y}}}}{2}}}{ \frac{1}{2}(-\beta+\sqrt{\beta^{2}+{\sigma_{\hat{\boldsymbol{Y}}}}})}
		%=2\frac{
		%	\big(\sqrt{\beta^{2}+\frac{\sigma_{\hat{\boldsymbol{Y}}}}{2}}-\beta\big)
		%	\big(\sqrt{\beta^{2}+\sigma_{\hat{\boldsymbol{Y}}}}+\beta\big)
		%}
		%{\sigma_{\hat{\boldsymbol{Y}}}}
		\geq \frac2{\sigma_{\hat{\boldsymbol{Y}}}}\Big(\beta^2+\frac{\sigma_{\hat{\boldsymbol{Y}}}}{2}-\beta^2\Big)=
		1,
		\]
		and thus the assumption on $\kappa$ implies 
		$\kappa^{2}+2\kappa\beta<\frac{\sigma_{\hat{\boldsymbol{Y}}}}{2}$.
	}
	On the other hand, we have
	\[
	\|\mathbb{E}[\hat{\boldsymbol{W}}\hat{\boldsymbol{W}}^{\top}]-\mathbb{E}[\hat{\boldsymbol{Y}}\hat{\boldsymbol{Y}}^{\top}]\|_{\mathrm{F}}\leq\|\hat{\boldsymbol{W}}\|_{[L^{2}(\Omega)]^{S}}\kappa+\kappa\beta\leq(\kappa+\beta)\kappa+\kappa\beta.
	\]
	Therefore, we obtain $\|\mathbb{E}[\hat{\boldsymbol{W}}\hat{\boldsymbol{W}}^{\top}]-\mathbb{E}[\hat{\boldsymbol{Y}}\hat{\boldsymbol{Y}}^{\top}]\|_{\mathrm{F}}<\frac{\sigma_{\hat{\boldsymbol{Y}}}}{2}$.
	From the well-known inequality $|\sigma_{\hat{\boldsymbol{Y}}}-\sigma_{\hat{\boldsymbol{W}}}|\leq\|\mathbb{E}[\hat{\boldsymbol{Y}}\hat{\boldsymbol{Y}}^{\top}]-\mathbb{E}[\hat{\boldsymbol{W}}\hat{\boldsymbol{W}}^{\top}]\|_{\mathrm{F}}$,
	see, e.g.\ \cite[Corollary 7.3.5]{Horn.R_Johnson_book_2013_2nd},
	we conclude
	\begin{equation}
	0<\frac{\sigma_{\hat{\boldsymbol{Y}}}}{2}<\sigma_{\hat{\boldsymbol{W}}}.\label{eq:lwbd-sigma-W}
	\end{equation}
	Next, we note that the identity
	\[
	(I-P_{\hat{\boldsymbol{W}}'})P_{\hat{\boldsymbol{W}}}g=(I-P_{\hat{\boldsymbol{W}}'})(\hat{\boldsymbol{W}}-\hat{\boldsymbol{W}}')^{\top}\big(\mathbb{E}[\hat{\boldsymbol{W}}\hat{\boldsymbol{W}}^{\top}]\big)^{-1}\mathbb{E}[\hat{\boldsymbol{W}}g]
	\]
	holds for any $g\in L^{2}(\Omega)$: indeed, we have
	\[
	(I-P_{\hat{\boldsymbol{W}}'})(\hat{\boldsymbol{W}}-\hat{\boldsymbol{W}}')^{\top}=(\hat{\boldsymbol{W}}-\hat{\boldsymbol{W}}')^{\top}-P_{\hat{\boldsymbol{W}}'}\hat{\boldsymbol{W}}^{\top}+(\hat{\boldsymbol{W}}')^{\top}=(I-P_{\hat{\boldsymbol{W}}'})\hat{\boldsymbol{W}}^{\top},
	\]
	but $\hat{\boldsymbol{W}}^{\top}\big(\mathbb{E}[\hat{\boldsymbol{W}}\hat{\boldsymbol{W}}^{\top}]\big)^{-1}\mathbb{E}[\hat{\boldsymbol{W}}g]=P_{\hat{\boldsymbol{W}}}g$.
	This type of identity was shown by Wedin in the finite dimensional
	setting, see \cite[(4.2)]{Wedin.P_1983_angles}. In view of this identity,
	the first inequality in \eqref{eq:Wedin-less-1} can be shown as
	\begin{align}
	\|(\hat{\boldsymbol{W}}-\hat{\boldsymbol{W}}'&)^{\top}  \big(\mathbb{E}[\hat{\boldsymbol{W}}\hat{\boldsymbol{W}}^{\top}]\big)^{-1}\mathbb{E}[\hat{\boldsymbol{W}}\boldsymbol{g}^{\top}]\|_{[L^{2}(\Omega)]^{S}}\nonumber \\
	& 
	\leq\|(\hat{\boldsymbol{W}}-\hat{\boldsymbol{W}}')\|_{[L^{2}(\Omega)]^{S}}
	\|\big(\mathbb{E}[\hat{\boldsymbol{W}}\hat{\boldsymbol{W}}^{\top}]\big)^{-1}\|_{{2}}
	\|\hat{\boldsymbol{W}}\|_{[L^{2}(\Omega)]^{S}}\|\boldsymbol{g}\|_{[L^{2}(\Omega)]^{S}}\nonumber \\
	& \leq\|(\hat{\boldsymbol{W}}-\hat{\boldsymbol{W}}')\|_{[L^{2}(\Omega)]^{S}}
	\frac{{1}}{\sigma_{\hat{\boldsymbol{W}}}}
	(\kappa+\beta)\|\boldsymbol{g}\|_{[L^{2}(\Omega)]^{S}}\nonumber \\
	& \leq\|(\hat{\boldsymbol{W}}-\hat{\boldsymbol{W}}')\|_{[L^{2}(\Omega)]^{S}}\frac{{2}}{\sigma_{\hat{\boldsymbol{Y}}}}(\kappa+\beta)\|\boldsymbol{g}\|_{[L^{2}(\Omega)]^{S}},\label{eq:Wedin-lem-last-touch}
	\end{align}
	where we used the assumption on $\hat{\boldsymbol{W}},\hat{\boldsymbol{W}}'$
	and \eqref{eq:lwbd-sigma-W}. Finally, we apply the inequality $\|(\hat{\boldsymbol{W}}-\hat{\boldsymbol{W}}')\|_{[L^{2}(\Omega)]^{S}}\leq2\kappa$
	to \eqref{eq:Wedin-lem-last-touch}. Then, noting that the assumption
	on $\kappa$ implies $\kappa^{2}+\beta\kappa<\frac{\sigma_{\hat{\boldsymbol{Y}}}}{{4}}$
	we have
	\[
	\|(\hat{\boldsymbol{W}}-\hat{\boldsymbol{W}}')\|_{[L^{2}(\Omega)]^{S}}
	\Big(\frac{{2}}{\sigma_{\hat{\boldsymbol{Y}}}}\Big)(\kappa+\beta)
	\leq\Big(\frac{{4}}{\sigma_{\hat{\boldsymbol{Y}}}}\Big)(\kappa^{2}+\kappa\beta)<1,
	\]
	which completes the proof.
	\end{proof}
\begin{lemma}\label{lem:lip-projY}
	Under the assumptions of Lemma~\ref{lem:Wedin-less-1},
	we have
	\[
	\|P_{\hat{\boldsymbol{W}}}-P_{\hat{\boldsymbol{W}}'}\|_{[L^{2}(\Omega)]^{S}\to[L^{2}(\Omega)]^{S}}\leq C_{\kappa,\beta,\sigma_{\hat{\boldsymbol{Y}}}}\|\hat{\boldsymbol{W}}-\hat{\boldsymbol{W}}'\|_{[L^{2}(\Omega)]^{S}}
	\]
	for any $\hat{\boldsymbol{W}},\hat{\boldsymbol{W}}'\in[L^{2}(\Omega)]^{S}$
	with $\|\hat{\boldsymbol{W}}-\hat{\boldsymbol{Y}}\|_{[L^{2}(\Omega)]^{S}}\leq\kappa$,
	$\|\hat{\boldsymbol{W}}'-\hat{\boldsymbol{Y}}\|_{[L^{2}(\Omega)]^{S}}\leq\kappa$,
	where $\kappa\in(0,\overline{\kappa}(\sigma_{\hat{\boldsymbol{Y}}},\beta))$
	and $C_{\kappa,\beta,\sigma_{\hat{\boldsymbol{Y}}}}$ are as in Lemma~\ref{lem:Wedin-less-1}.
\end{lemma}
\begin{proof}
	In view of Lemma~\ref{lem:Wedin-less-1}, it suffices
	to show $\|P_{\hat{\boldsymbol{W}}}-P_{\hat{\boldsymbol{W}}'}\|_{[L^{2}(\Omega)]^{S}\to[L^{2}(\Omega)]^{S}}=\|(I-P_{\hat{\boldsymbol{W}}'})P_{\hat{\boldsymbol{W}}}\|_{[L^{2}(\Omega)]^{S}\to[L^{2}(\Omega)]^{S}}$.
	We will invoke a perturbation result on pairs of projections, \cite[Lemma 221]{Kato.T_1958_two_proj},
	see also \cite[Theorem I.6.34]{Kato.T_book_1995_reprint}. In this
	regard, first we will show the following identity of finite dimensional
	vector subspaces
	\begin{equation}
	\mathrm{Im}(P_{\hat{\boldsymbol{W}}'}|_{\mathrm{Im}(P_{\hat{\boldsymbol{W}}})}):=P_{\hat{\boldsymbol{W}}'}\big(P_{\hat{\boldsymbol{W}}}([L^{2}(\Omega)]^{S})\big)=P_{\hat{\boldsymbol{W}}'}([L^{2}(\Omega)]^{S})=:\mathrm{Im}(P_{\hat{\boldsymbol{W}}'}).\label{eq:image-identity}
	\end{equation}
	Since $\mathrm{Im}(P_{\hat{\boldsymbol{W}}'}|_{\mathrm{Im}(P_{\hat{\boldsymbol{W}}})})\subset\mathrm{Im}(P_{\hat{\boldsymbol{W}}'})$,
	it suffices to show that $\mathrm{Im}(P_{\hat{\boldsymbol{W}}'}|_{\mathrm{Im}(P_{\hat{\boldsymbol{W}}})})$
	cannot be a proper subspace of $\mathrm{Im}(P_{\hat{\boldsymbol{W}}'})$.
	To see this, we will verify that the dimension of $\mathrm{Im}(P_{\hat{\boldsymbol{W}}'}|_{\mathrm{Im}(P_{\hat{\boldsymbol{W}}})})$
	is the same as $\mathrm{Im}(P_{\hat{\boldsymbol{W}}'})$. We now note
	that, in view of \eqref{eq:lwbd-sigma-W} in the proof of Lemma~\ref{lem:Wedin-less-1},
	we have
	\[
	\mathrm{dim}(\mathrm{Im}(P_{\hat{\boldsymbol{W}}'}))=S=\mathrm{dim}(\mathrm{Im}(P_{\hat{\boldsymbol{W}}})).
	\]
	Therefore, if the linear operator $P_{\hat{\boldsymbol{W}}'}|_{\mathrm{Im}(P_{\hat{\boldsymbol{W}}})}\colon\mathrm{Im}(P_{\hat{\boldsymbol{W}}})\to\mathrm{Im}(P_{\hat{\boldsymbol{W}}'}|_{\mathrm{Im}(P_{\hat{\boldsymbol{W}}})})$
	is a vector space isomorphism, then we have $\mathrm{dim}(\mathrm{Im}(P_{\hat{\boldsymbol{W}}'}|_{\mathrm{Im}(P_{\hat{\boldsymbol{W}}})}))=S$,
	and thus \eqref{eq:image-identity} will follow. It suffices to show
	the injectivity. For any $\boldsymbol{x}=P_{\hat{\boldsymbol{W}}}\boldsymbol{x}\in\mathrm{Im}(P_{\hat{\boldsymbol{W}}})$,
	with {$d:=\|(I-P_{\hat{\boldsymbol{W}}'})P_{\hat{\boldsymbol{W}}}\|_{[L^{2}(\Omega)]^{S}\to[L^{2}(\Omega)]^{S}}$} we have
	\begin{align*}
	\|\boldsymbol{x}-P_{\hat{\boldsymbol{W}}'}\boldsymbol{x}\|_{[L^{2}(\Omega)]^{S}} & =\|P_{\hat{\boldsymbol{W}}}\boldsymbol{x}-P_{\hat{\boldsymbol{W}}'}P_{\hat{\boldsymbol{W}}}\boldsymbol{x}\|_{[L^{2}(\Omega)]^{S}}
	\leq d\|\boldsymbol{x}\|_{[L^{2}(\Omega)]^{S}},
	\end{align*}
	where from \eqref{eq:Wedin-less-1} we have $d<1$. Thus, we get 
	$\|\boldsymbol{x}\|_{[L^{2}(\Omega)]^{S}}{\leq}\frac{1}{1-d}\|P_{\hat{\boldsymbol{W}}'}\boldsymbol{x}\|_{[L^{2}(\Omega)]^{S}}$,
	which shows the injectivity. Hence we have \eqref{eq:image-identity}.
	
	Finally, in view of \cite[i) Lemma 221]{Kato.T_1958_two_proj},
	we have
	\[
	\|P_{\hat{\boldsymbol{W}}}-P_{\hat{\boldsymbol{W}}'}\|_{[L^{2}(\Omega)]^{S}\to[L^{2}(\Omega)]^{S}}=\|(I-P_{\hat{\boldsymbol{W}}'})P_{\hat{\boldsymbol{W}}}\|_{[L^{2}(\Omega)]^{S}\to[L^{2}(\Omega)]^{S}},
	\]
	and the statement follows from Lemma~\ref{lem:Wedin-less-1}.
	\end{proof}
Next, we derive a local Lipschitz continuity of the inverse of the Gram matrix
$Z_{\hat{\boldsymbol{U}}}=\langle\hat{\boldsymbol{U}},\hat{\boldsymbol{U}}^{\top}\rangle$.
\begin{lemma}
	\label{lem:pre-lip-ZU}
	Suppose that  $\hat{\boldsymbol{U}},\hat{\boldsymbol{U}}'\in[\mathcal{H}]^{S}$
	are linearly independent and that for some $\tilde{\alpha}>0$ we have  $\max\{\|\hat{\boldsymbol{U}}\|_{[\mathcal{H}]^{S}},\|\hat{\boldsymbol{U}}'\|_{[\mathcal{H}]^{S}}\}\leq\tilde{\alpha}$. 
	Then, it holds
	\[{\|Z_{\hat{\boldsymbol{U}}}^{-1}-Z_{\hat{\boldsymbol{U}}'}^{-1}\|_{{2}}\leq C_{\tilde{\alpha},S}(\|Z_{\hat{\boldsymbol{U}}}^{-1}\|_{{2}}^{2}+\|Z_{\hat{\boldsymbol{U}'}}^{-1}\|_{{2}}^{2})\|\hat{\boldsymbol{U}}-\hat{\boldsymbol{U}}'\|_{[\mathcal{H}]^{S}}}\]
	with a constant $C_{\tilde{\alpha},S}>0$.
\end{lemma}

\begin{proof}
	For components $\hat{U}_{j}$, $\hat{U}_{k}$ of $\hat{\boldsymbol{U}}$;
	and $\hat{U}'_{j}$, $\hat{U}'_{k}$ of $\hat{\boldsymbol{U}}'$,
	we have
	\[
	|\langle\hat{U}_{j},\hat{U}{}_{k}\rangle-\langle\hat{U}'_{j},\hat{U}'_{k}\rangle|\leq
	\max\{\|\hat{U}_{k}\|_{\mathcal{H}},\|\hat{U}'_{j}\|_{\mathcal{H}}\}(\|\hat{U}{}_{j}-\hat{U}'_{j}\|_{\mathcal{H}}+\|\hat{U}{}_{k}-\hat{U}'_{k}\|_{\mathcal{H}}),
	\]
	and thus there exists a constant $C'_{\tilde{\alpha},S}$ depending
	on $S$ such that 
	$
	\|Z_{\hat{\boldsymbol{U}}}-Z_{\hat{\boldsymbol{U}}'}\|_{2}
	\leq
	\|Z_{\hat{\boldsymbol{U}}}-Z_{\hat{\boldsymbol{U}}'}\|_{\mathrm{F}}\leq C'_{\tilde{\alpha},S}\|\hat{\boldsymbol{U}}-\hat{\boldsymbol{U}}'\|_{[\mathcal{H}]^{S}}$.

	Noting that the matrix $Z_{\hat{\boldsymbol{U}}}$ is non-singular
	when $\hat{\boldsymbol{U}}$ is linear independent, we recall that
	the Fréchet derivative of the mapping $\mathbb{{R}}^{S\times S}\ni B\mapsto B^{-1}=:\mathrm{Inv}(B)\in\mathbb{{R}}^{S\times S}$
	at $B\in\mathbb{{R}}^{S\times S}$ acting on $W\in\mathbb{R}^{S\times S}$
	is given by $D\mathrm{Inv}(B)[W]=-B^{-1}WB^{-1}$ (see, for example
	\cite[Appendix A.5]{Absil.P_etal_2008_book}). Then, with the notation
	\[
	\|D\mathrm{Inv}(Z_{\hat{\boldsymbol{U}}})\|_{\mathbb{R}^{S\times S}\to\mathbb{R}^{S\times S}}
	:=\max_{W\in\mathbb{R}^{S\times S}:\|W\|_{{2}}=1}\|Z_{\hat{\boldsymbol{U}}}^{-1}WZ_{\hat{\boldsymbol{U}}}^{-1}\|_{{2}},
	\]
	in view of \cite[Corollary 3.2]{Coleman.R_2012_book} we
	have
	\begin{align*}
	&\|Z_{\hat{\boldsymbol{U}}}^{-1}  -Z_{\hat{\boldsymbol{U}}'}^{-1}\|_{{2}}
	=\|\mathrm{Inv}(Z_{\hat{\boldsymbol{U}}})-\mathrm{Inv}(Z_{\hat{\boldsymbol{U}}'})\|_{{2}}\\
	&\leq
	\sup\Big\{\|D\mathrm{Inv}(\tilde{Z})\|_{\mathbb{R}^{S\times S}\to\mathbb{R}^{S\times S}}\,\Big|\,\tilde{Z}=rZ_{\hat{\boldsymbol{U}}}+(1-r)Z_{\hat{\boldsymbol{U}}'},\:r\in[0,1]\Big\}\\
	&\phantom{a little space}\times\|Z_{\hat{\boldsymbol{U}}}-Z_{\hat{\boldsymbol{U}}'}\|_{{2}}\\
	%
	%&=  \sup\Big\{\max_{W\in\mathbb{R}^{S\times S}:\|W\|_{{2}}=1}
	%\|\tilde{Z}^{-1}W\tilde{Z}^{-1}\|_{{2}}\,\Big|\,\tilde{Z}=rZ_{\hat{\boldsymbol{U}}}+(1-r)Z_{\hat{\boldsymbol{U}}'},\:r\in[0,1]\Big\}
	%	\|Z_{\hat{\boldsymbol{U}}}-Z_{\hat{\boldsymbol{U}}'}\|_{{2}}\\
	&\leq  \sup\Big\{\|\tilde{Z}^{-1}\|_{{2}}^{2}\,\Big|\,\tilde{Z}=rZ_{\hat{\boldsymbol{U}}}+(1-r)Z_{\hat{\boldsymbol{U}}'},\:r\in[0,1]\Big\}\|Z_{\hat{\boldsymbol{U}}}-Z_{\hat{\boldsymbol{U}}'}\|_{{2}}.
	\end{align*}
	Now, for $r\in[0,1]$ given, since $Z_{\hat{\boldsymbol{U}}}^{-1}$ and $Z_{\hat{\boldsymbol{U}}'}^{-1}$ are symmetric positive definite, from \cite{Moore.M.H_1973_matrix_convex} we have 
	$\boldsymbol{c}^\top
	\big(rZ_{\hat{\boldsymbol{U}}}+(1-r)Z_{\hat{\boldsymbol{U}}'}\big)^{-1}
	\boldsymbol{c}
	\leq 
	r\boldsymbol{c}^\top
	Z_{\hat{\boldsymbol{U}}}^{-1}
	\boldsymbol{c}
	+
	(1-r)\boldsymbol{c}^\top
	Z_{\hat{\boldsymbol{U}}'}^{-1}
	\boldsymbol{c}$ for any $\boldsymbol{c}\in\mathbb{R}^S$, 
	and thus
	$\big\|\big(rZ_{\hat{\boldsymbol{U}}}+(1-r)Z_{\hat{\boldsymbol{U}}'}\big)^{-1}\big\|_2
	\leq \|Z_{\hat{\boldsymbol{U}}}^{-1}\|_2+\|Z_{\hat{\boldsymbol{U}}'}^{-1}\|_2
	$. Therefore, we obtain
	\[
	\|Z_{\hat{\boldsymbol{U}}}^{-1}-Z_{\hat{\boldsymbol{U}}'}^{-1}\|_{2}
	\leq 
	2(\|Z_{\hat{\boldsymbol{U}}}^{-1}\|_{{2}}^{2}+\|Z_{\hat{\boldsymbol{U}}'}^{-1}\|_{{2}}^{2})\|Z_{\hat{\boldsymbol{U}}}-Z_{\hat{\boldsymbol{U}}'}\|_{{2}}.
	\]
	It follows that
	\[
	\|Z_{\hat{\boldsymbol{U}}}^{-1}-Z_{\hat{\boldsymbol{U}}'}^{-1}\|_{{2}} \leq2(\|Z_{\hat{\boldsymbol{U}}}^{-1}\|_{{2}}^{2}+\|Z_{\hat{\boldsymbol{U}}'}^{-1}\|_{{2}}^{2})C'_{\tilde{\alpha},S}\|\hat{\boldsymbol{U}}-\hat{\boldsymbol{U}}'\|_{[\mathcal{H}]^{S}},
	\]
	which completes the proof.
	\end{proof}
As a consequence, we obtain the following.
\begin{lemma}
	\label{lem:lip-ZU}Suppose that $\hat{\boldsymbol{U}}\in[\mathcal{H}]^{S}$
	is linearly independent and $\|\hat{\boldsymbol{U}}\|_{[\mathcal{H}]^{S}}\leq\alpha$
	for some $\alpha>0$. {Then, we have $\|Z_{\hat{\boldsymbol{U}}}^{-1}\|\leq\gamma$
		for some $\gamma>0$. Further, there exists a constant $C_{\alpha,S}>0$
		that is independent of the position $\hat{\boldsymbol{U}}$ and $R=R(\hat{\boldsymbol{U}})\in(0,1]$
		such that
		\[
		\|Z_{\hat{\boldsymbol{V}}}^{-1}-Z_{\hat{\boldsymbol{V}}'}^{-1}\|\leq\gamma^{2}C_{\alpha,S}\|\hat{\boldsymbol{V}}-\hat{\boldsymbol{V}}'\|_{[\mathcal{H}]^{S}}
		\]
	}holds for any $\hat{\boldsymbol{V}},\hat{\boldsymbol{V}}'\in[\mathcal{H}]^{S}$,
	with $\|\hat{\boldsymbol{V}}-\hat{\boldsymbol{U}}\|_{[\mathcal{H}]^{S}}\leq R$,
	$\|\hat{\boldsymbol{V}}'-\hat{\boldsymbol{U}}\|_{[\mathcal{H}]^{S}}\leq R$.
\end{lemma}
\begin{proof}
	Since $\|Z_{\hat{\boldsymbol{U}}}^{-1}\|_{{2}}\leq\gamma$, for $R=R(\hat{\boldsymbol{U}})\in(0,1]$
	small enough we have $\|Z_{\hat{\boldsymbol{V}}}^{-1}\|_{{2}}\leq2\gamma$
	for all $\hat{\boldsymbol{V}}\in[\mathcal{H}]^{S}$ such that $\|\hat{\boldsymbol{V}}-\hat{\boldsymbol{U}}\|_{[\mathcal{H}]^{S}}\leq R$.
	Such $\hat{\boldsymbol{V}}$ satisfies $\|\hat{\boldsymbol{V}}\|_{[\mathcal{H}]^{S}}\leq\alpha+R\leq\alpha+1$.
	Thus, with $\tilde{\alpha}:=\alpha+1$ and $C_{\alpha,S}:=8C_{\tilde{\alpha},S}$
	in Lemma~\ref{lem:pre-lip-ZU} the statement follows.
	\end{proof}
Lemmata~\ref{lem:lip-projY} and \ref{lem:lip-ZU} established above
give the following local Lipschitz continuity of the non-linear term
$G$ we need.
\begin{proposition}
	\label{prop:loc-Lip-G}Let Assumption~\ref{assu:lip-F-H} hold. Suppose
	that we have $\|Z_{{\hat{\boldsymbol{U}}}}^{-1}\|_{{2}}\leq\gamma$ 
	for {${\hat{\boldsymbol{U}}}\in[\mathcal{H}]^{S}$, and that
		${\hat{\boldsymbol{Y}}}\in[L^{2}(\Omega)]^{S}$ is $L^{2}(\Omega)$-orthonormal.}
	Then, $G:\mathcal{X}\rightarrow\mathcal{X}$ defined in \eqref{eq:def-G}
	satisfies the assumption of Proposition~\ref{prop:gen-mild-local} 
	for this $(\hat{\boldsymbol{U}},\hat{\boldsymbol{Y}})$ with
	a constant depending also on $\gamma$.
\end{proposition}
\begin{proof}
	Let {$\alpha\geq\|\hat{\boldsymbol{U}}\|_{[\mathcal{H}]^{S}}$
		and $\beta\geq\|\hat{\boldsymbol{Y}}\|_{[L^{2}(\Omega)]^{S}}=\sqrt{S}$}
	be given. First, from Assumption~\ref{assu:lip-F-H} we have $\|[G_{1}(\hat{\boldsymbol{Y}})](\hat{\boldsymbol{U}})\|_{[\mathcal{H}]^{S}}\leq{C_{\alpha,\beta}}$,
	and further, together with $\|Z_{\hat{\boldsymbol{U}}}^{-1}\|_{{2}}\leq\gamma$
	we have $\|[G_{2}(\hat{\boldsymbol{U}})](\hat{\boldsymbol{Y}})\|_{[\mathcal{H}]^{S}}\leq{C_{\alpha,\beta,\gamma}}$.
	It now suffices to show
	\begin{align*}
	\|\left[G_{1}(\boldsymbol{W})\right]&(\boldsymbol{V})-\left[G_{1}\left(\boldsymbol{W}'\right)\right](\boldsymbol{V}')\|_{[\mathcal{H}]^{S}}\\
	&
	\leq C_{\alpha,\beta}\Big(\|\boldsymbol{V}-\boldsymbol{\boldsymbol{V}}'\|_{[\mathcal{H}]^{S}}^{2}+\|\boldsymbol{W}-\boldsymbol{W}'\|_{[L^{2}(\Omega)]^{S}}^{2}\Big)^{1/2},
	\end{align*}
	and
	\begin{align*}
	\|\left[G_{2}(\boldsymbol{V})\right]&(\boldsymbol{W})-\left[G_{2}(\boldsymbol{V}')\right](\boldsymbol{W}')\|_{[L^{2}(\Omega)]^{S}}\\
	&\leq C'_{\alpha,\beta,\gamma}\Big(\|\boldsymbol{V}-\boldsymbol{V}'\|_{[\mathcal{H}]^{S}}^{2}+\|\boldsymbol{W}-\boldsymbol{W}'\|_{[L^{2}(\Omega)]^{S}}^{2}\Big)^{1/2}
	\end{align*}
	in closed balls centred at $\hat{\boldsymbol{U}}$, and $\hat{\boldsymbol{Y}}$,
	respectively, with a radius ${r=r(\hat{\boldsymbol{U}},\hat{\boldsymbol{Y}})}$.
	The first inequality can be checked from Assumption~\ref{assu:lip-F-H}.
	The second inequality follows from Lemmata~\ref{lem:lip-projY} and
	\ref{lem:lip-ZU} by letting $r<\min\{R(\hat{\boldsymbol{U}}),\overline{\kappa}(1,\beta)\}$,
	where $\overline{\kappa}(1,\beta)$ is as in Lemma~\ref{lem:Wedin-less-1}.
	\end{proof}
\section{Existence and regularity\label{sec:Existence-and-Regularity}}
Using the Lipschitz continuity established in the previous section,
we will now show the existence of the Dual DO solution on the
maximal interval. We start with local existence of the mild solution
$(\boldsymbol{U},\boldsymbol{Y})$ of the problem~\eqref{eq:Cauchy}.
Further, we will see that under suitable conditions, such solution is indeed the Dual DO solution in the strong, and furthermore in the classical sense {(Definitions~\ref{def:Dual-DO-strong} and \ref{def:Dual-DO-classical})}.
Hence, from the equivalence established in Section~\ref{subsec:equiv-DLR-DO},
we will be able to conclude that $\boldsymbol{U}^{\top}\boldsymbol{Y}$
is the solution of the original equation~\eqref{eq:original-eq-with-proj}.
\begin{proposition}[mild, local]\label{prop:mild-local} Let Assumptions \ref{assu:Lam}
	and \ref{assu:lip-F-H} hold. Suppose that the initial condition $\boldsymbol{U}_{0}\in[\mathcal{H}]^{S}$
	is linearly independent in $\mathcal{H}$, and $\boldsymbol{Y}_{0}\in[L^{2}(\Omega)]^{S}$
	is orthonormal in $L^{2}(\Omega)$. Then, there exists $t^{*}{=t^{*}(\boldsymbol{U}_{0},\boldsymbol{Y}_{0})}>0$
	such that the mild solution of the abstract Cauchy problem \eqref{eq:Cauchy}
	uniquely exists on $[0,t^{*}]$.
\end{proposition}

\begin{proof}
	In view of Proposition \ref{prop:loc-Lip-G}, the statement follows
	from Proposition\ \ref{prop:gen-mild-local}.
	\end{proof}
A regularity of the initial condition gives us the existence of the
strong solution.
\begin{proposition}[strong, local]
	\label{prop:Cauchy-stong-local}
	Let Assumptions \ref{assu:Lam}
	and \ref{assu:lip-F-H} hold. Suppose further that the initial condition
	$\boldsymbol{U}_{0}\in[\mathcal{H}]^{S}$ is linearly independent
	in $\mathcal{H}$, and $\boldsymbol{Y}_{0}\in[L^{2}(\Omega)]^{S}$
	is orthonormal in $L^{2}(\Omega)$. 
	Furthermore, suppose that $(\boldsymbol{U}_{0},\boldsymbol{Y}_{0})\in D(A)$.
	Then, the mild solution obtained in Proposition~\ref{prop:mild-local}
	is the strong solution of the abstract Cauchy problem \eqref{eq:Cauchy}.
\end{proposition}

\begin{proof}
	In view of \cite[Theorem 6.1.6]{Pazy.A_1983_book}, the statement
	follows from Proposition~\ref{prop:mild-local}.
	\end{proof}
The above strong solution is actually the Dual DO approximation in
the strong sense.
\begin{corollary}[Dual DO-strong, local]
	\label{cor:DO-stong-local} Let the assumptions of
	Proposition\ \ref{prop:Cauchy-stong-local} hold. Then, the strong
	solution $(\boldsymbol{U}(t),\boldsymbol{Y}(t))$ of the abstract
	Cauchy problem \eqref{eq:Cauchy} uniquely exists on a non-empty interval
	$[0,t^{*}]$. The solution $\boldsymbol{U}(t)$ stays linearly independent
	on $[0,t^{*}]$ and the solution $\boldsymbol{Y}(t)$ is orthonormal
	in $L^{2}(\Omega)$ for $t\in[0,t^{*}]$ {and satisfies the gauge condition $\mathbb{E}[\dot{\boldsymbol{Y}}(t) \boldsymbol{Y}(t)^\top ]=0$ for almost every $t\in[0,t^*]$}. Hence, the Dual DO approximation
	uniquely exists in the strong sense on $[0,t^{*}]$.
\end{corollary}

\begin{proof}
	It suffices to show the linear independence of $\boldsymbol{U}(t)$
	and the orthonormality of $\boldsymbol{Y}(t)$. 
	But the solution of
	the abstract Cauchy problem \eqref{eq:Cauchy} established in Proposition~\ref{prop:Cauchy-stong-local}
	exists only on an interval $[0,t^{*}]$ on which the inverse Gram
	matrix $Z_{\boldsymbol{U}}^{-1}$ is well defined. Hence, on this
	interval, $\boldsymbol{U}(t)$ is linear independent.
	
	To see the orthonormality, first note that, from the absolute continuity of $\boldsymbol{Y}(t)$, the function $\mathbb{E}[Y_{j}Y_{k}]$ is absolutely continuous on $[0,T]$. 
	But following the same argument as \eqref{eq:implied-gauge}, we have $\frac{\mathrm{d}}{\mathrm{d}t}\mathbb{E}[Y_{j}Y_{k}]=
	\mathbb{E}[\dot{Y}_{j}Y_{k}]+\mathbb{E}[Y_{j}\dot{Y}_{k}]=0$ a.e.\ on $[0,T]$. 
	{Therefore, from the orthonormality of the initial condition, for every $t\in[0,t^{*}]$ we 
		have $\mathbb{E}[Y_{j}(t)Y_{k}(t)]-\delta_{jk}=\int _0^t0\mathrm{d}t=0$, where $\delta_{jk}=1$ only if $j=k$, and $0$ otherwise. Hence, $\boldsymbol{Y}(t)$ is orthonormal for all $t\in[0,t^*]$.}
	\end{proof}
With a further regularity of $F$, we obtain the Dual DO approximation
in the classical sense.
\begin{corollary}[Dual DO-classical, local]\label{cor:classical-local} Suppose that Assumptions
	\ref{assu:Lam}, \ref{assu:lip-F-H} and \ref{assu:lip-F-D(Lam)}
	are satisfied. Suppose further that the initial condition $\boldsymbol{U}_{0}\in D(\Lambda)$
	is linearly independent in $\mathcal{H}$, and $\boldsymbol{Y}_{0}\in[L^{2}(\Omega)]^{S}$
	is orthonormal in $L^{2}(\Omega)$. 
	Then, there exists $t^{*}>0$ such that the Dual DO approximation
	uniquely exists in the classical sense on $[0,t^{*}]$.
\end{corollary}
\begin{proof}
	We first observe that $G:[\mathcal{V}]^{S}\oplus[L^{2}(\Omega)]^{S}\rightarrow[\mathcal{V}]^{S}\oplus[L^{2}(\Omega)]^{S}$
	is locally Lipschitz, where $\mathcal{V}$ is the Hilbert space $D_{\mathcal{H}}(\Lambda)$
	equipped with the graph norm.
	Further, we note that $(\mathrm{e}^{tA})_{t\ge0}$ is a $C_{0}$ semigroup
	on $[\mathcal{V}]^{S}\oplus[L^{2}(\Omega)]^{S}$.%\footnote{Let $T(t)$, $t\ge0$ be a $C_{0}$ semigroup on a Banach space $X$ generated by a closed operator $\mathcal{A}$. Let $D(\mathcal{A})$ be equipped with the graph norm, which we denote $Y$. Pazy page 190 claims $T(t)\colon Y\to Y$, $t\ge0$ is a $C_{0}$ semigroup on $Y$. 
	%\begin{enumerate}
	%\item Boundedness of $T(t)$: Let $x\in D(\mathcal{A})$. We have $T(0)x=x\in D(\mathcal{A})$. For $t>0$ we have $T(t)x\in D(\mathcal{A})$, and further the boundedness of follows from $\mathcal{A}T(t)x=T(t)\mathcal{A}x$ (Pazy' book Theorem 1.2.4.): 
	%\[
	%\|T(t)x\|_{Y}=\|T(t)x\|+\|\mathcal{A}T(t)x\|\leq\|T(t)\|_{X\to X}\|x\|+\|T(t)\|_{X\to X}\|\mathcal{A}x\|.
	%\]
	%\item Strongly continuity: Noting that $\mathcal{A}x\in X$ for $x\in D(\mathcal{A})$,
	%we have
	%\[
	%\|T(t)x-x\|+\|\mathcal{A}T(t)x-\mathcal{A}x\|=\|T(t)x-x\|+\|T(t)\mathcal{A}x-\mathcal{A}x\|\to0\qquad\text{as }t\to0.
	%\]
	%\end{enumerate}
	%} 
	With these in mind, we see that a result analogous to Proposition~\ref{prop:gen-mild-local}
	holds in $[\mathcal{V}]^{S}\oplus[L^{2}(\Omega)]^{S}$.
	Then, in view of the discussion in \cite[pages 190--191]{Pazy.A_1983_book},
	the statement follows from the similar argument as in the proof of
	Corollary\ \ref{cor:DO-stong-local}.
	\end{proof}
We now extend the solution to the maximal time interval in the strong,
and the classical sense. For both, we need bounds for the solution
in terms of a suitable (semi)norm. We start with the following bound.
\begin{lemma}
	\label{lem:UY-bdd-UH}
	\begin{sloppypar}
		Let Assumption~\ref{assu:stab} hold. Suppose
		that the strong solution $(\boldsymbol{U}(t),\boldsymbol{Y}(t))$
		of the abstract Cauchy problem~\eqref{eq:Cauchy} exists on $[0,t^{*}]$.
		Then, we have $\|\boldsymbol{Y}(t)\|_{[L^{2}(\Omega)]^{S}}=\|\boldsymbol{Y}_{0}\|_{[L^{2}(\Omega)]^{S}}$
		for all $t\in[0,t^{*}]$. Furthermore, we have
		\[
		\|\boldsymbol{U}(t)\|_{[\mathcal{H}]^{S}}\leq(\sqrt{2C_{\Lambda,F}}t^{1/2}+\|\boldsymbol{U}_{0}\|_{[\mathcal{H}]^{S}}\mathrm{e}^{C_{\Lambda,F}t})\quad\text{ for all }t\in[0,t^{*}].
		\]
	\end{sloppypar}
\end{lemma}

\begin{proof}
	Following the same argument as \eqref{eq:implied-gauge}, we
	have $\mathbb{E}\big[Y_{j}\frac{\partial}{\partial t}Y_{j}\big]=0$
	a.e.~$[0,t^{*}]$, $j=1,\dots,S$. Hence, $\|\boldsymbol{Y}(t)\|_{[L^{2}(\Omega)]^{S}}$
	is constant a.e.~$[0,t^{*}]$. Then, the continuity of $t\mapsto\|\boldsymbol{Y}(t)\|_{[L^{2}(\Omega)]^{S}}$
	implies the first statement. Next, almost everywhere in $[0,t^{*}]$
	we have
	\[
	\frac{\partial}{\partial t}\boldsymbol{U}^{\top}=\Lambda(\boldsymbol{U}^{\top})+\mathbb{E}\left[F(u_{S})\boldsymbol{Y}^{\top}\right]=\mathbb{E}\left[\Lambda(\boldsymbol{U}^{\top}\boldsymbol{Y})\boldsymbol{Y}^{\top}\right]+\mathbb{E}\left[F(u_{S})\boldsymbol{Y}^{\top}\right],
	\]
	where each component is in $\mathcal{H}$, and hence {for $j=1,\dots,S$, we have}
	\[
	\langle\frac{\partial}{\partial t}U_{j},U_{j}\rangle=\mathbb{E}[\langle\Lambda(u_{S})],U_{j}Y_{j}\rangle]+\mathbb{E}\left[\langle F(u_{S}),U_{j}Y_{j}\rangle\right].
	\]
	Hence, because of Assumption \ref{assu:stab} and the orthonormality
	of $\{Y_{j}\}$ we have
	\begin{align*}
	\frac{\mathrm{d}}{\mathrm{d}t}\sum_{j=1}^{S}\|U_{j}\|^{2} & =2\mathbb{E}[\langle\Lambda(u_{S})+F(u_{S}),u_{S}\rangle]\\
	& \leq2C_{\Lambda,F}(1+\|u_{S}\|_{L^{2}(\Omega;\mathcal{H})}^{2})\leq2C_{\Lambda,F}\Big(1+\sum_{j=1}^{S}\|U_{j}\|^{2}\Big),
	\end{align*}
	and thus \[\sum_{j=1}^{S}\|U_{j}(t)\|^{2}\leq(2C_{\Lambda,F}t+\sum_{j=1}^{S}\|U_{j}(0)\|^{2})+2C_{\Lambda,F}\int_{0}^{t}\sum_{j=1}^{S}\|U_{j}(s)\|^{2}\mathrm{d}s.\] Therefore, the Gronwall's inequality implies
	\[
	\sum_{j=1}^{S}\|U_{j}(t)\|^{2}\leq\Big(2C_{\Lambda,F}t+\sum_{j=1}^{S}\|U_{j}(0)\|^{2}\Big)\mathrm{e}^{2C_{\Lambda,F}t},
	\]
	and thus the second statement holds for almost every $t$. Noting
	that the mapping $t\mapsto[\|\boldsymbol{U}(t)\|_{[\mathcal{H}]^{S}}-\text{(\ensuremath{\sqrt{2C_{\Lambda,F}}t^{1/2}}+\ensuremath{\|\boldsymbol{U}}(0)\ensuremath{\|_{[\mathcal{H}]^{S}}})\ensuremath{\mathrm{e}^{C_{\Lambda,F}t}}}]$
	is continuous, this is true for every $t\in[0,t^{*}]$.%\footnote{Let $g(t)=[\|\boldsymbol{U}(t)\|_{[\mathcal{H}]^{S}}-\text{(\ensuremath{\sqrt{2C_{\Lambda,F}}t^{1/2}}+\ensuremath{\|\boldsymbol{U}}(0)\ensuremath{\|_{[\mathcal{H}]^{S}}})\ensuremath{\mathrm{e}^{t}}}]$. Suppose $g(r)>0$ at some $r\in[0,t^{*}]$. Then, since $g$ is continuous, there exists $\delta\in(0,t^{*}/2)$ such that for any $t\in B_{\delta}(r)\cap[0,t^{*}]$ we have $|g(t)-g(r)|<g(r)/2$. If $g(t)\ge g(r)$, then $g(t)>0$; if not, then
	%\[
	%0<g(r)/2<g(t),
	%\]
	%and thus $g>0$ on $B_{\delta}(r)\cap[0,t^{*}]$, which has a positive
	%measure $\delta\le|B_{\delta}(r)\cap[0,t^{*}]|\le2\delta$, contradiction.}
	\end{proof}
We are ready to establish the existence of Dual DO solution in
the strong sense until $\boldsymbol{U}$ becomes linearly dependent.
\begin{theorem}[Dual DO-strong, maximal]\label{thm:strong-maximal}Suppose Assumptions \ref{assu:Lam},
	\ref{assu:lip-F-H}, and \ref{assu:stab} are satisfied. Suppose further
	that the initial condition $\boldsymbol{U}_{0}\in[\mathcal{H}]^{S}$
	is linearly independent in $\mathcal{H}$, and $\boldsymbol{Y}_{0}\in[L^{2}(\Omega)]^{S}$
	is orthonormal in $L^{2}(\Omega)$. Further, suppose that $(\boldsymbol{U}_{0},\boldsymbol{Y}_{0})\in D(A)$.
	Then, there exists $t_{\max}>0$ such that the Dual DO approximation
	uniquely exists in the strong sense on $\left[0,t_{\max}\right)$.
	The approximation can be extended in time until the Gram matrix $Z_{\boldsymbol{U}}$
	of $\boldsymbol{U}$ becomes singular: we have either 
	\[
	t_{\max}=\infty,\quad\text{or}\quad\lim_{t\uparrow t_{\max}}\|Z_{\boldsymbol{U}(t)}^{-1}\|_{{2}}=\infty.
	\]
\end{theorem}
\begin{proof}
	Under the condition $(\boldsymbol{U}_{0},\boldsymbol{Y}_{0})\in D(A)$,
	it suffices to show the maximality of the mild solution. We show that
	$t_{\max}<\infty$ implies $\lim_{t\uparrow t_{\max}}\|Z_{\boldsymbol{U}(t)}^{-1}\|_{{2}}=\infty$.
	In this regard, we first show $\limsup_{t\uparrow t_{\max}}\|Z_{\boldsymbol{U}(t)}^{-1}\|_{{2}}=\infty$. 
	{We argue by contradiction and assume $t_{\max}<\infty$ and $\limsup_{t\uparrow t_{\max}}\|Z_{\boldsymbol{U}(t)}^{-1}\|_{{2}}<\infty$. }
	{T}hen we have 
	\[
	\sup_{t\in[t_{\max}-\delta,t_{\max})}\|Z_{\boldsymbol{U}(t)}^{-1}\|_{{2}}<\infty\quad\text{ for sufficiently small }\delta>0.
	\]
	Thus, since $\max_{t\in[0,t_{\max}-\delta]}\|Z_{\boldsymbol{U}(t)}^{-1}\|_{{2}}<\infty$
	for any $0<\delta<t_{\max}$, with a constant $K>0$ we have $\|Z_{\boldsymbol{U}(t)}^{-1}\|_{{2}}<K$
	for all $t\in[0,t_{\max})$. Now Lemma~\ref{lem:UY-bdd-UH} implies
	$\|\boldsymbol{Y}(t)\|_{[L^{2}(\Omega)]^{S}}=\sqrt{S}$
	and
	\[
	\|\boldsymbol{U}(t)\|_{[\mathcal{H}]^{S}}\leq\alpha_{\max}:=(\sqrt{2C_{\Lambda,F}}t_{\max}^{1/2}+\|\boldsymbol{U}_{0}\|_{[\mathcal{H}]^{S}}\mathrm{e}^{C_{\Lambda,F}t_{\max}}),\quad t\ensuremath{\in}[0,t_{\max}),
	\]
	and thus in view of Proposition~\ref{prop:loc-Lip-G} we have $\|[G_{1}(\boldsymbol{Y}(s))](\boldsymbol{U}(s))\|_{[\mathcal{H}]^{S}}\leq C_{\alpha_{\max},S}$
	for any $s\in[0,t_{\max})$. If $0<t<t'<t_{\max}$ then letting
	$K_{\Lambda}=\sup_{r\in[0,t_{\max}]}\|\mathrm{e}^{r\Lambda}\|$ we
	have
	\begin{align*}
	\|\boldsymbol{U}(t')&-\boldsymbol{U}(t)\|_{[\mathcal{H}]^{S}} \\
	& \leq\|\mathrm{e}^{t'\Lambda}\boldsymbol{U}(0)-\mathrm{e}^{t\Lambda}\boldsymbol{U}(0)\|_{[\mathcal{H}]^{S}}+\bigg\|\int_{t}^{t'}\mathrm{e}^{(t'-s)\Lambda}\left[G_{1}(\boldsymbol{Y}(s))\right](\boldsymbol{U}(s))\mathrm{d}s\bigg\|_{[\mathcal{H}]^{S}}\\
	& \qquad+\bigg\|\int_{0}^{t}(\mathrm{e}^{(t'-s)\Lambda}-\mathrm{e}^{(t-s)\Lambda})\left[G_{1}(\boldsymbol{Y}(s))\right](\boldsymbol{U}(s))\mathrm{d}s\bigg\|_{[\mathcal{H}]^{S}}\\
	%& %\leq\|\mathrm{e}^{t'\Lambda}\boldsymbol{U}(0)-\mathrm{e}^{t\Lambda}\boldsymbol{U}(0)\|_{[\mathcal{H}]^{S}}+(t'-t)K_{\Lambda}C_{\alpha_{\max},S}\\
	%& \qquad+\int_{0}^{t_{\max}}\|(\mathrm{e}^{(t'-s)\Lambda}-\mathrm{e}^{(t-s)\Lambda})\left[G_{1}(\boldsymbol{Y}(s))\right](\boldsymbol{U}(s))\|_{[\mathcal{H}]^{S}}\mathrm{d}s\\
	& \leq\|\mathrm{e}^{t'\Lambda}\boldsymbol{U}(0)-\mathrm{e}^{t\Lambda}\boldsymbol{U}(0)\|_{[\mathcal{H}]^{S}}+(t'-t)K_{\Lambda}C_{\alpha_{\max},S}\\
	& \qquad+\int_{0}^{t_{\max}}\|\mathrm{e}^{(t-s)\Lambda}\|\|(\mathrm{e}^{(t'-t)\Lambda}-I)\left[G_{1}(\boldsymbol{Y}(s))\right](\boldsymbol{U}(s))\|_{[\mathcal{H}]^{S}}\mathrm{d}s.
	\end{align*}
	From $\|\left[G_{1}(\boldsymbol{Y}(s))\right](\boldsymbol{U}(s))\|_{[\mathcal{H}]^{S}}\leq C_{\alpha_{\max},S}$,
	the dominated convergence theorem implies that the right hand side
	of
	\begin{align*}
	\int_{0}^{t_{\max}}\|\mathrm{e}^{(t-s)\Lambda}\| & \|(\mathrm{e}^{(t'-t)\Lambda}-I)\left[G_{1}(\boldsymbol{Y}(s))\right](\boldsymbol{U}(s))\|_{[\mathcal{H}]^{S}}\mathrm{d}s\\
	& \leq\int_{0}^{t_{\max}}K_{\Lambda}\|(\mathrm{e}^{(t'-t)\Lambda}-I)\left[G_{1}(\boldsymbol{Y}(s))\right](\boldsymbol{U}(s))\|_{[\mathcal{H}]^{S}}\mathrm{d}s
	\end{align*}
	tends to zero as $t,t'$ tend to $t_{\max}$. Hence, $\|\boldsymbol{U}(t')-\boldsymbol{U}(t)\|_{[\mathcal{H}]^{S}}\to0$
	as $t,t'\to t_{\max}$. Therefore, $\boldsymbol{U}$ admits a continuous
	extension $\lim_{t\uparrow t_{\max}}\boldsymbol{U}(t)=\boldsymbol{U}(t_{\max})$.
	This allows us to extend $Z_{\boldsymbol{U}(t)}^{-1}$ to $[0,t_{\max}]$.
	Indeed, Lemma~\ref{lem:pre-lip-ZU} implies
	\[
	\|Z_{\boldsymbol{U}(t')}^{-1}-Z_{\boldsymbol{U}(t)}^{-1}\|_{{2}}\leq2C_{\alpha_{\max},S}{K^{2}}\|\boldsymbol{U}(t')-\boldsymbol{U}(t)\|_{[\mathcal{H}]^{S}},
	\]
	and thus we have $\lim_{t\uparrow t_{\max}}Z_{\boldsymbol{U}(t)}^{-1}={Z^{*}\in\mathbb{R}^{S\times S}}$
	with ${\|Z^{*}\|_{{2}}\leq K}$, but we must have $Z^{*}=Z_{\boldsymbol{U}(t_{\max})}^{-1}$.
	Similarly, noting that $\|Z_{\boldsymbol{U}(s)}^{-1}\|_{{2}}<K$ implies
	\[\|[G_{2}(\boldsymbol{Y}(s))](\boldsymbol{U}(s))\|_{[L^{2}(\Omega)]^{S}}\leq C_{\alpha_{\max},S,K},\]
	we see that $\lim_{t\uparrow t_{\max}}\boldsymbol{Y}(t)=\boldsymbol{Y}(t_{\max})$
	exists, {and from Corollary~\ref{cor:DO-stong-local} we have
		$\mathbb{E}[\boldsymbol{Y}(t_{\max})\boldsymbol{Y}(t_{\max})^{\top}]=I$}.
	But in view of Proposition~\ref{prop:loc-Lip-G} these consequences
	imply that we can extend the solution beyond $t_{\max}$, which contradicts
	the maximality of $[0,t_{\max})$. Hence, $\limsup_{t\uparrow t_{\max}}\|Z_{\boldsymbol{U}(t)}^{-1}\|_{{2}}=\infty$. 
	
	To conclude the proof we will show 
	\[
	\lim_{t\uparrow t_{\max}}\|Z_{\boldsymbol{U}(t)}^{-1}\|_{{2}}=\infty.
	\]
	If this is false, then there exist a sequence $t_{n}\uparrow t_{\max}$
	and $\gamma>0$ such that $\|Z_{\boldsymbol{U}(t_{n})}^{-1}\|_{{2}}\leq\gamma$
	for all $n\ge0$. But since $\limsup_{t\uparrow t_{\max}}\|Z_{\boldsymbol{U}(t)}^{-1}\|_{{2}}=\infty$
	there is a sequence $s_{k}\uparrow t_{\max}$ such that $\|Z_{\boldsymbol{U}(s_{k})}^{-1}\|_{{2}}\geq\gamma+1$
	for all $k\ge0$. We take a subsequence $(s_{k_{n}})_{n}$ so that
	$t_{n}<s_{k_{n}}$ for all $n$. From the continuity of $t\mapsto\|Z_{\boldsymbol{U}(t)}^{-1}\|_{{2}}$
	on $[t_{n},s_{k_{n}}]$, there exists $h_{n}\in[0,s_{k_{n}}-t_{n}]$
	such that $\|Z_{\boldsymbol{U}(t_{n}+h_{n})}^{-1}\|_{{2}}=\gamma+1$.
	Now, from Lemma~\ref{lem:pre-lip-ZU} we have for any $n\geq0$
	\[
	1
	\!\leq\!
	\|Z_{\boldsymbol{U}(t_{n}+h_{n})}^{-1}\!\!\;\|_{{2}}-\|Z_{\boldsymbol{U}(t_{n})}^{-1}\|_{{2}}
	\!\leq\! C_{\alpha_{\max},S}{(2\gamma^{2}\!\!\;+2\gamma+1)}\|\boldsymbol{U}(t_{n}+h_{n})-\boldsymbol{U}(t_{n})\|_{\!\!\;[\mathcal{H}]^{S}},
	\]
	which is absurd since $|h_{n}|\leq|s_{k_{n}}-t_{\max}|+|t_{\max}-t_{n}|\to0$
	as $n\to\infty$ and $\boldsymbol{U}$ is continuous on $[0,t_{\max})$.
	Hence, the proof is complete.
	\end{proof}
Under a stronger assumption on the non-linear term $F$, we obtain
the following bound for $\|\Lambda\boldsymbol{U}(t)\|_{[\mathcal{H}]^{S}}$.
This bound will be used to establish the existence in the classical
sense on the maximal interval. 
\begin{lemma}
	\label{lem:UY-bdd-U-D(Lam)}Let Assumptions \ref{assu:Lam} and \ref{assu:glob-lin-F-D(A)}
	hold. Suppose that the classical solution $(\boldsymbol{U}(t),\boldsymbol{Y}(t))$
	of the abstract Cauchy problem \eqref{eq:Cauchy} exists on $[0,t^{*}]$
	for some $t^{*}>0$. Then, we have 
	\[
	\|\Lambda\boldsymbol{U}(t)\|_{[\mathcal{H}]^{S}}\leq\text{{\ensuremath{K_{\Lambda}}}}\ensuremath{(}\|\ensuremath{\Lambda}\boldsymbol{U}(0)\|_{[\mathcal{H}]^{S}}+tC_{F}\sqrt{S})\mathrm{e}^{\text{{\ensuremath{K_{\Lambda}}}}C_{F}t}\quad\text{ for }t\in[0,t^{*}],
	\]
	where the constant $C_{F}>0$ is from Assumption \ref{assu:glob-lin-F-D(A)}.
\end{lemma}

\begin{proof}
	We have
	\begin{align*}
	\Lambda\boldsymbol{U}(t) & =\Lambda e^{t\Lambda}\boldsymbol{U}(0)+\int_{0}^{t}\Lambda e^{(t-\tau)\Lambda}\mathbb{E}[F(\boldsymbol{U}(\tau)^{\top}\boldsymbol{Y}(\tau))\boldsymbol{Y}(\tau)]\mathrm{d}\tau\\
	& =e^{t\Lambda}\Lambda\boldsymbol{U}(0)+\int_{0}^{t}e^{(t-\tau)\Lambda}\Lambda\mathbb{E}[F(\boldsymbol{U}(\tau)^{\top}\boldsymbol{Y}(\tau))\boldsymbol{Y}(\tau)]\mathrm{d}\tau\\
	& =e^{t\Lambda}\Lambda\boldsymbol{U}(0)+\int_{0}^{t}e^{(t-\tau)\Lambda}\mathbb{E}[\Lambda F(\boldsymbol{U}(\tau)^{\top}\boldsymbol{Y}(\tau))\boldsymbol{Y}(\tau)]\mathrm{d}\tau,
	\end{align*}
	and thus, noting that Assumption \ref{assu:Lam} implies $\|e^{s\Lambda}\|_{[\mathcal{H}]^{S}\to[\mathcal{H}]^{S}}\leq\text{{\ensuremath{K_{\Lambda}}}}$,
	$s\ge0$, we get
	\[
	\|\Lambda\boldsymbol{U}(t)\|_{[\mathcal{H}]^{S}}\leq{K_{\Lambda}}\|\Lambda\boldsymbol{U}(0)\|_{[\mathcal{H}]^{S}}+{K_{\Lambda}}\int_{0}^{t}\mathbb{E}[\|\Lambda F(\boldsymbol{U}(\tau)^{\top}\boldsymbol{Y}(\tau))\boldsymbol{Y}(\tau)\|_{[\mathcal{H}]^{S}}]\mathrm{d}\tau.
	\]
	From $\mathbb{E}[|\boldsymbol{Y}(\tau)|^{2}]=\boldsymbol{1}$ and
	Assumption~\ref{assu:glob-lin-F-D(A)}, we have
	\begin{align*}
	\mathbb{E}[\|\Lambda F(\boldsymbol{U}(\tau)^{\top}\boldsymbol{Y}(\tau))\boldsymbol{Y}(\tau)\|_{[\mathcal{H}]^{S}}] & \leq\|\Lambda F(\boldsymbol{U}(\tau)^{\top}\boldsymbol{Y}(\tau))\|_{L^{2}(\Omega;\mathcal{H})}\sqrt{S}\\
	& \leq C_{F}(1+\|\Lambda\boldsymbol{U}(\tau)\|_{[\mathcal{H}]^{S}})\sqrt{S}.
	\end{align*}
	Hence, we get \[\|\Lambda\boldsymbol{U}(t)\|_{[\mathcal{H}]^{S}}\leq{K_{\Lambda}}(\|\Lambda\boldsymbol{U}(0)\|_{[\mathcal{H}]^{S}}+tC_{F}\sqrt{S})+{K_{\Lambda}}C_{F}\int_{0}^{t}\|\Lambda\boldsymbol{U}(\tau)\|_{[\mathcal{H}]^{S}}\mathrm{d}\tau.\]
	Then, applying the Gronwall's inequality completes the proof.
	\end{proof}
\begin{theorem}[Dual DO-classical, maximal]\label{thm:classical-maximal} Suppose Assumptions
	\ref{assu:Lam}--\ref{assu:glob-lin-F-D(A)} are satisfied. Suppose
	further that the initial condition $\boldsymbol{U}_{0}\in[\mathcal{H}]^{S}$
	is linearly independent in $\mathcal{H}$, and $\boldsymbol{Y}_{0}\in[L^{2}(\Omega)]^{S}$
	is orthonormal in $L^{2}(\Omega)$. Further, suppose that $(\boldsymbol{U}_{0},\boldsymbol{Y}_{0})\in D(A)$.
	Then, there exists $t_{\max}>0$ such that the Dual DO approximation
	uniquely exists in the classical sense on $\left[0,t_{\max}\right)$.
	The approximation can be extended in time until the Gram matrix $Z_{\boldsymbol{U}}$
	of $\boldsymbol{U}$ becomes singular: we have either 
	\[
	t_{\max}=\infty,\quad\text{or}\quad\lim_{t\uparrow t_{\max}}\|Z_{\boldsymbol{U}(t)}^{-1}\|_{{2}}=\infty.
	\]
\end{theorem}
\begin{proof}
	Our argument is analogous to the proof of Theorem~\ref{thm:strong-maximal},
	but here we consider the parameter equation in $[\mathcal{V}]^{S}\oplus[L^{2}(\Omega)]^{S}$.
	The only difference thus is the equation for $\boldsymbol{U}$, but
	Lemmata~\ref{lem:UY-bdd-UH} and \ref{lem:UY-bdd-U-D(Lam)} give
	a bound for $\|\boldsymbol{U}(t)\|_{[\mathcal{V}]^{S}}$, $t\in[0,t_{\max})$,
	and Assumption~\ref{assu:glob-lin-F-D(A)} gives a bound for $\|\left[G_{1}(\boldsymbol{Y}(s))\right](\boldsymbol{U}(s))\|_{[\mathcal{V}]^{S}}$,
	$s\in[0,t_{\max})$. 
	Further, we have \[\sup_{r\in[0,t_{\max}]}\|\mathrm{e}^{r\Lambda}\|_{[\mathcal{V}]^{S}\to[\mathcal{V}]^{S}}\leq\sup_{r\in[0,t_{\max}]}\|\mathrm{e}^{r\Lambda}\|_{[\mathcal{H}]^{S}\to[\mathcal{H}]^{S}}\leq K_{\Lambda},\]
	and $\|\mathrm{e}^{s\Lambda}-\mathrm{e}^{t\Lambda}\|_{[\mathcal{V}]^{S}\to[\mathcal{V}]^{S}}\leq\|\mathrm{e}^{s\Lambda}-\mathrm{e}^{t\Lambda}\|_{[\mathcal{H}]^{S}\to[\mathcal{H}]^{S}}$. %\footnote{$\|\mathrm{e}^{r\Lambda}x\|_{[\mathcal{H}]^{S}}^{2}+\|\Lambda\mathrm{e}^{r\Lambda}x\|_{[\mathcal{H}]^{S}}^{2}=\|\mathrm{e}^{r\Lambda}x\|_{[\mathcal{H}]^{S}}^{2}+\|\mathrm{e}^{r\Lambda}\Lambda x\|_{[\mathcal{H}]^{S}}^{2}\leq\|\mathrm{e}^{r\Lambda}\|_{[\mathcal{H}]^{S}\to[\mathcal{H}]^{S}}^{2}(\|x\|_{[\mathcal{H}]^{S}}^{2}+\|\Lambda x\|_{[\mathcal{H}]^{S}}^{2})\leq K_{\Lambda}^{2}\|x\|_{[\mathcal{V}]^{S}}^{2}$}
	Noting that $(\boldsymbol{U},\boldsymbol{Y})$ is also a mild solution
	in $\mathcal{X}$, the extension of $Z_{\boldsymbol{U}}^{-1}$ can
	be established. Hence, we see that the mild solution in $[\mathcal{V}]^{S}\oplus[L^{2}(\Omega)]^{S}$
	exists on $[0,t_{\max})$, and that if $t_{\max}<\infty$ then $\lim_{t\uparrow t_{\max}}\|Z_{\boldsymbol{U}(t)}^{-1}\|_{{2}}=\infty$.
	But in view of \cite[Corollary 4.2.6, Theorem 6.1.7]{Pazy.A_1983_book}
	this is a classical solution, and thus the proof is complete. 
	\end{proof}
We are now interested in continuing the DLR approximation $u_{S}$
beyond the maximal time $t_{\max}$. A difficulty arising is the full
rank condition imposed on $M_{S}$: at $t_{\max}$ the spatial basis
becomes linearly dependent, and thus the solution will not stay in
$M_{S}$. But from a practical point of view this should be favourable---roughly
speaking, at the maximal time a smaller number of spatial basis is
sufficient to capture the same information as $\boldsymbol{U}$ does.
This observation motivates us to leave $M_{S}$: to extend the approximation
beyond $t_{\max}$ we consider the extension to $t_{\max}$ in the
ambient space $L^{2}(\Omega;\mathcal{H})$. To do so, we go back to
the original formulation {\eqref{eq:original-eq-with-proj}}.
Then, upon extending the solution to $t_{\max}$, one can re-start
from $t_{\max}$ with a suitable decomposition as the initial condition.
\begin{proposition}
	\begin{sloppypar}
		Let the assumptions of Theorem~\ref{thm:classical-maximal}
		hold. Then, with the classical solution $(\boldsymbol{U},\boldsymbol{Y})$
		as in Theorem~\ref{thm:classical-maximal}, $u_{S}=\boldsymbol{U}^{\top}\boldsymbol{Y}\colon[0,t_{\max})\to L^{2}(\Omega;\mathcal{H})$
		is Lipschitz continuous. Thus, $u_{S}$ admits a unique continuous
		extension to $[0,t_{\max}]$. 
	\end{sloppypar}
\end{proposition}

\begin{proof}
	\begin{comment}
	From Lemma~\ref{lem:UY-bdd-UH} we have
	\begin{align*}
	\Big\|\frac{\mathrm{d}}{\mathrm{d}t}u_{S}\Big\|_{L^{2}(\Omega;\mathcal{H})} & \leq\sum_{j=1}^{S}\Big(\Big\| Y_{j}\frac{\mathrm{d}}{\mathrm{d}t}U_{j}\Big\|_{L^{2}(\Omega;\mathcal{H})}+\Big\| U_{j}\frac{\mathrm{d}}{\mathrm{d}t}Y_{j}\Big\|_{L^{2}(\Omega;\mathcal{H})}\Big)\\
	& \leq\sum_{j=1}^{S}\Big(\Big\|\frac{\mathrm{d}}{\mathrm{d}t}U_{j}\Big\|_{\mathcal{H}}+C_{t_{\max},\Lambda,F,\|\boldsymbol{U}_{0}\|}\Big\|\frac{\mathrm{d}}{\mathrm{d}t}Y_{j}\Big\|_{L^{2}(\Omega)}\Big),
	\end{align*}
	and thus on any interval $[0,\tilde{t}]\subset[0,t_{\max})$ we have
	$\frac{\mathrm{d}}{\mathrm{d}t}u_{S}\in L^{1}([0,\tilde{t}];L^{2}(\Omega;\mathcal{H}))$.
	Letting
	\[
	v_{S}(t):=u_{S}(0)+\int_{0}^{t}\frac{\mathrm{d}}{\mathrm{d}t}u_{S}(\tau)\mathrm{d}\tau,\quad t\in[0,\tilde{t}],
	\]
	$v_{S}$ is absolutely continuous on $[0,\tilde{t}]$. But from the
	continuity of $\frac{\mathrm{d}}{\mathrm{d}t}u_{S}$ on $(0,\tilde{t})$,
	we have $\frac{\mathrm{d}}{\mathrm{d}t}v_{S}=\frac{\mathrm{d}}{\mathrm{d}t}u_{S}$
	on $(0,\tilde{t})$, and thus $v_{S}-u_{S}$ is constant on $(0,\tilde{t})$
	\cite[Theorems 3.9 and 3.5]{Coleman.R_2012_book}. But since $v_{S}(0)-u_{S}(0)=0$,
	the continuity of $u_{S}$ and $v_{S}$ implies $v_{S}=u_{S}$ on
	$[0,\tilde{t}]$.
	\end{comment}
	%Therefore, for any $0\leq t'<t<t_{\max}$, we have
	Noting that $u_S$ is absolutely continuous on $[0,t]\subset [t_{\max})$, for any $0\leq t'<t<t_{\max}$ we have
	\[
	\|u_{S}(t)-u_{S}(t')\|_{L^{2}(\Omega;\mathcal{H})}\leq\int_{t'}^{t}\big(\|\Lambda u_{S}(r)\|_{L^{2}(\Omega;\mathcal{H})}+\|F(u_{S}(r))\|_{L^{2}(\Omega;\mathcal{H})}\big)\mathrm{d}r.
	\]
	But from Lemma~\ref{lem:UY-bdd-U-D(Lam)} and Assumption~\ref{assu:lip-F-H}
	we have
	\[
	\|\Lambda u_{S}(r)\|_{L^{2}(\Omega;\mathcal{H})}\leq\sum_{j=1}^{S}\|\Lambda U_{j}(r)\|\leq\sqrt{S}(K_{\Lambda}\|\Lambda\boldsymbol{U}(0)\|_{[\mathcal{H}]^{S}}+t_{\max}C_{F}\sqrt{S}),
	\]
	and $
	\|F(u_{S}(r))\|_{L^{2}(\Omega;\mathcal{H})}\leq C_{t_{\max},K_{\Lambda},F}
	$ for some constant $C_{t_{\max},K_{\Lambda},F}\!>0$.  
	Hence, we obtain
	\begin{align*}
	\|u_{S}(t)&-u_{S}(t')\|_{L^{2}(\Omega;\mathcal{H})}\\
	&\leq(t-t')\big(\sqrt{S}(K_{\Lambda}\|\Lambda\boldsymbol{U}(0)\|_{[\mathcal{H}]^{S}}+t_{\max}C_{F}\sqrt{S})+C_{t_{\max},K_{\Lambda},F}\big),
	\end{align*}
	and thus $u_{S}$ admits a continuous extension $u_{S}(t)\to u^{*}=:u_{S}(t_{\max})$
	as $t\uparrow t_{\max}$.
	\end{proof}

Finally, we will show the existence of smooth parametrisation given a smooth curve $[0,T]\ni t\mapsto u_S(t)\in M_S$, announced in Proposition~\ref{prop:smooth-SVD}. 
Our argument is similar to the existence proofs in this section thus far. We note, however, that since $\dot{u}_S$ is assumed to be given, and no unbounded operator is involved, although the equation depends on time via  $\dot{u}_S$ the proof is simpler.
\begin{proof}[Proof of Proposition~\ref{prop:smooth-SVD}]
	Consider the following  ordinary differential equation in $\mathbb{R}^{S\times S}\oplus[\mathcal{H}]^{S}\oplus[L^{2}(\Omega)]^{S}$: 
	\begin{align*}
	\dot{\Sigma}&=\mathbb{E}[\langle\tilde{\boldsymbol{V}},\dot{u}_{S}\boldsymbol{W}^{\top}\rangle]\\
	\dot{\tilde{\boldsymbol{V}}}^{\top}\Sigma&=\mathbb{E}[\dot{u}_{S}\boldsymbol{W}^{\top}]-\tilde{\boldsymbol{V}}^{\top}\langle\tilde{\boldsymbol{V}}\mathbb{E}[\dot{u}_{S}\boldsymbol{W}^{\top}]\rangle=:(I-P_{\tilde{\boldsymbol{V}}})\big(\mathbb{E}[\dot{u}_{S}\boldsymbol{W}^{\top}]\big)\\
	\Sigma\dot{\boldsymbol{W}}
	&
	=\langle\tilde{\boldsymbol{V}},\dot{u}_{S}\rangle-\mathbb{E}[\langle\tilde{\boldsymbol{V}},\dot{u}_{S}\rangle\boldsymbol{W}^{\top}]\boldsymbol{W}=:(I-P_{\boldsymbol{W}})\langle\tilde{\boldsymbol{V}},\dot{u}_{S}\rangle.
	\end{align*}
	If this equation
	has a solution $(\Sigma,\tilde{\boldsymbol{V}},\boldsymbol{W})$ with
	the desired smoothness, then the statement follows.
	
	But from $\dot{u}_{S}\in L^{1}([0,T];L^{2}(\Omega;\mathcal{H}))$
	and the local Lipschitz continuity of the projection-operator-valued
	mappings, see Lemma~\ref{lem:lip-projY}, there exists a unique solution locally in time. 
	Moreover,
	any solution $\tilde{\boldsymbol{V}}$ and $\boldsymbol{W}$ must
	preserve the orthogonality, see the proof of Corollary~\ref{cor:DO-stong-local}.  Furthermore, Lemma~\ref{lem:uS-SVD}
	guarantees the stability and the invertibility of $\Sigma$
	on $[0,T]$. 
	Thus, following an  argument similar to that of the proof of  Theorem~\ref{thm:strong-maximal}, we observe that the solution $(\Sigma,\tilde{\boldsymbol{V}},\boldsymbol{W})$
	can be uniquely extended to $[0,T]$. Now the proof is complete.
	
	The proof for the continuous differentiability is analogous.
	\end{proof}
\section{Conclusions\label{sec:conclusions}}

We established the existence of the dynamical low rank (DLR) approximation
for random semi-linear evolutionary equations on the maximal interval.
A key was to consider an equivalent formulation, the Dual DO formulation.
After showing that the Dual DO formulation is indeed equivalent, we
showed the unique existence of the solution in the strong and classical
sense, by invoking results for the abstract Cauchy problem in the
vector spaces. Further, we considered a continuation of the DLR approximation
beyond the maximal time interval.

\section*{Acknowledgements}
We thank Eva Vidli\v{c}kov\'a for helpful discussions.  
This work has been supported by the Swiss National Science Foundation under the Project n.\ 172678 ``Uncertainty Quantification techniques for PDE constrained optimization and random evolution equations''. The authors also acknowledge the support from the Center for ADvanced MOdeling Science (CADMOS).

%\printbibliography
\bibliographystyle{plain}
\bibliography{mybib.bib}
\end{document}